\documentclass[12pt]{article}
\usepackage{listings}
\usepackage{fullpage}
\usepackage{url}
\usepackage{amsmath}
\usepackage{setspace}
\usepackage{sectsty}
\usepackage{mathdots}

\newcommand{\simple}{1 }

\newcommand{\gen}{2 }
\newcommand{\genns}{2}

\newcommand{\examplex}{3 }
\newcommand{\examplexns}{3}

\newcommand{\wfN}{4 }

\newcommand{\wfNNN}{5 }
\newcommand{\wfNNNns}{5}

\newcommand{\ramsey}{6 }

\newcommand{\trramsey}{7 }

\newcommand{\matone}{8 }

\newcommand{\mattwo}{9 }

\newcommand{\matthree}{10 }

\newcommand{\llex}{<_{\rm lex}}

\newcommand{\mat}[1]{ \begin{pmatrix}  #1 \end{pmatrix}}
\newcommand{\mip}{MIP }
\newcommand{\mipit}{{\it MIP }}
\newcommand{\mipns}{MIP}
\newcommand{\compseg}{computational segment }
\newcommand{\compsegns}{computational segment}

\newcommand{\compsegs}{computational segment\ }
\newcommand{\While}{{\bf While }}
\newcommand{\inp}{{\bf Input}}

\renewcommand{\and}{ \hbox{  and }  }

\newcommand{\cc}[1]{c_{#1}}
\newcommand{\ii}[1]{i_{#1}}
\newcommand{\clos}{{\rm clos}}

\newcommand{\w}{w}
\newcommand{\x}{x}
\newcommand{\y}{y}
\newcommand{\z}{z}

\newcommand{\Erdos}{Erd{\"o}s }

\newcommand{\ang}[1]{\langle#1\rangle}

\newcommand{\cvg}{\downarrow}

\newcommand{\goes}{\rightarrow}

\newcommand{\R}{{\sf R}}

\newcommand{\Z}{{\sf Z}}
\newcommand{\nat}{{\sf N}}

\newcommand{\xvec}[1]{\ifcase 3{#1} {\ang {x_1,x_2,x_3} } \else 
\ifcase 4{#1} {\ang{x_1,x_2,x_3,x_4}} \else {\ang {x_1,\ldots,x_{#1}}}\fi\fi}
\newcommand{\yvec}[1]{\ifcase 3{#1} {\ang {y_1,y_2,y_3} } \else 
\ifcase 4{#1} {\ang{y_1,y_2,y_3,y_4}} \else {\ang {y_1,\ldots,y_{#1}}}\fi\fi}
\newcommand{\zvec}[1]{\ifcase 3{#1} {\ang {z_1,z_2,z_3} } \else 
\ifcase 4{#1} {\ang{z_1,z_2,z_3,z_4}} \else {\ang {z_1,\ldots,z_{#1}}}\fi\fi}
\newcommand{\vecc}[2]{\ifcase 3{#2} {\ang { {#1}_1,{#1}_2,{#1}_3 } } \else
\ifcase 4{#1} {\ang { {#1}_1,{#1}_2,{#1}_3,{#1}_{4} } }
\else {\ang { {#1}_1,\ldots,{#1}_{#2}}}\fi\fi}
\newcommand{\veccd}[3]{\ifcase 3{#2} {\ang { {#1}_{{#3}1},{#1}_{{#3}2},{#1}_{{#3}3} } } \else
\ifcase 4{#1} {\ang { {#1}_{{#3}1},{#1}_{{#3}2},{#1}_{#3}3},{#1}_{{#3}4} }
\else {\ang { {#1}_{{#3}1},\ldots,{#1}_{{#3}{#2}}}}\fi\fi}
\newcommand{\veccz}[2]{\ifcase 3{#2} {\ang { {#1}_0,{#1}_2,{#1}_3 } } \else
\ifcase 4{#1} {\ang { {#1}_0,{#1}_2,{#1}_3,{#1}_{4} } }
\else {\ang { {#1}_0,\ldots,{#1}_{#2}}}\fi\fi}
\newcommand{\xve}[1]{\ifcase 3{#1} {x_1,x_2,x_3} \else 
\ifcase 4{#1} {x_1,x_2,x_3,x_4} \else {x_1,\ldots,x_{#1}}\fi\fi}
\newcommand{\yve}[1]{\ifcase 3{#1} {y_1,y_2,y_3} \else 
\ifcase 4{#1} {y_1,y_2,y_3,y_4} \else {y_1,\ldots,y_{#1}}\fi\fi}
\newcommand{\zve}[1]{\ifcase 3{#1} {z_1,z_2,z_3} \else 
\ifcase 4{#1} {z_1,z_2,z_3,z_4} \else {z_1,\ldots,z_{#1}}\fi\fi}
\newcommand{\ve}[2]{\ifcase 3#2 {{#1}_1,{#1}_2,{#1}_3} \else
\ifcase 4#2 {{#1}_1,{#1}_2,{#1}_3,{#1}_{4}}
\else {{#1}_1,\ldots,{#1}_{#2}}\fi\fi}
\newcommand{\ved}[3]{\ifcase 3#2 {{#1}_{{#3}1},{#1}_{{#3}2},{#1}_{{#3}3}} \else
\ifcase 4#2 {{#1}_{{#3}1},{#1}_{{#3}2},{#1}_{{#3}3},{#1}_{{#3}4}}
\else {{#1}_{{#3}1},\ldots,{#1}_{{#3}{#2}}}\fi\fi}
\newcommand{\fuve}[3]{
\ifcase 3#2
{{#3}({#1}_1),{#3}({#1}_2,{#3}({#1}_3)} \else
\ifcase 4#2
{{#3}({#1}_1),{#3}({#1}_2),{#3}({#1}_3),{#3}({#1}_4)}
\else
{{#3}({#1}_1),\ldots,{#3}({#1}_{#2})}\fi\fi}
\newcommand{\setmathchar}[1]{\ifmmode#1\else$#1$\fi}
\newcommand{\vlist}[2]{%
	\setmathchar{%
		\compound#2\one{#2}\two
		\ifcompound
			({#1}_1,\ldots,{#1}_{#2})
		\else
			\ifcat N#2
				({#1}_1,\ldots,{#1}_{#2})
			\else
				\ifcase#2
					({#1}_0)\or
					({#1}_1)\or
					({#1}_1,{#1}_2)\or 
					({#1}_1,{#1}_2,{#1}_3)\or
					({#1}_1,{#1}_2,{#1}_3,{#1}_4)\else 
					({#1}_1,\ldots,{#1}_{#2})
				\fi
			\fi
		\fi}}

\newif\ifcompound
\def\compound#1\one#2\two{%
	\def\one{#1}
	\def\two{#2}
	\if\one\two
		\compoundfalse
	\else
		\compoundtrue
	\fi}

\newcommand{\xwe}[1]{\ifcase 3{#1} {x_1\wedge x_2\wedge x_3} \else 
\ifcase 4{#1} {x_1\wedge x_2\wedge x_3\wedge x_4} \else {x_1\wedge \cdots \wedge
x_{#1}}\fi\fi}
\newcommand{\we}[2]{\ifcase 3#2 {\ang { {#1}_1\wedge {#1}_2\wedge {#1}_3 } } \else
\ifcase 4{#1} {\ang { {#1}_1\wedge {#1}_2\wedge {#1}_3\wedge {#1}_{4} } }
\else {\ang { {#1}_1\wedge \cdots\wedge {#1}_{#2}}}\fi\fi}

\newcommand{\st}{\mathrel{:}}

\newcommand{\into}{\rightarrow}

\newcommand{\es}{\emptyset}

\newcommand{\floor}[1]{\left\lfloor{#1}\right\rfloor}

\newcommand{\union}{\cup}

\newcommand{\monus}{\;\raise.5ex\hbox{{${\buildrel
    \ldotp\over{\hbox to 6pt{\hrulefill}}}$}}\;}

\newcommand{\infinity}{\infty}

\newcounter{savenumi}

\newtheorem{theoremfoo}{Theorem}[section] 
\newenvironment{theorem}{\pagebreak[1]\begin{theoremfoo}}{\end{theoremfoo}}

\newtheorem{lemmafoo}[theoremfoo]{Lemma}
\newenvironment{lemma}{\pagebreak[1]\begin{lemmafoo}}{\end{lemmafoo}}
\newtheorem{conjecturefoo}[theoremfoo]{Conjecture}

\newtheorem{conventionfoo}[theoremfoo]{Convention}
\newenvironment{convention}{\pagebreak[1]\begin{conventionfoo}\rm}{\end{conventionfoo}}

\newtheorem{porismfoo}[theoremfoo]{Porism}

\newtheorem{gamefoo}[theoremfoo]{Game}

\newtheorem{corollaryfoo}[theoremfoo]{Corollary}

\newtheorem{openfoo}[theoremfoo]{Open Problem}

\newtheorem{exercisefoo}{Exercise}

\newcommand{\fig}[1] 
{
 \begin{figure}
 \begin{center}
 \input{#1}
 \end{center}
 \end{figure}
}

\newtheorem{potanafoo}[theoremfoo]{Potential Analogue}

\newtheorem{notefoo}[theoremfoo]{Note}
\newenvironment{note}{\pagebreak[1]\begin{notefoo}\rm}{\end{notefoo}}

\newtheorem{notabenefoo}[theoremfoo]{Nota Bene}

\newtheorem{nttn}[theoremfoo]{Notation}
\newenvironment{notation}{\pagebreak[1]\begin{nttn}\rm}{\end{nttn}}

\newtheorem{empttn}[theoremfoo]{Empirical Note}

\newtheorem{examfoo}[theoremfoo]{Example}
\newenvironment{example}{\pagebreak[1]\begin{examfoo}\rm}{\end{examfoo}}

\newtheorem{dfntn}[theoremfoo]{Def}
\newenvironment{definition}{\pagebreak[1]\begin{dfntn}\rm}{\end{dfntn}}

\newtheorem{propositionfoo}[theoremfoo]{Proposition}

\newenvironment{proof}
    {\pagebreak[1]{\narrower\noindent {\bf Proof:\quad\nopagebreak}}}{\QED}
\newenvironment{sketch}
    {\pagebreak[1]{\narrower\noindent {\bf Proof sketch:\quad\nopagebreak}}}{\QED}

\newcommand{\yyskip}{\penalty-50\vskip 5pt plus 3pt minus 2pt}
\newcommand{\blackslug}{\hbox{\hskip 1pt
        \vrule width 4pt height 8pt depth 1.5pt\hskip 1pt}}
\newcommand{\QED}{{\penalty10000\parindent 0pt\penalty10000
        \hskip 8 pt\nolinebreak\blackslug\hfill\lower 8.5pt\null}
        \par\yyskip\pagebreak[1]}

\newcommand{\BBB}{{\penalty10000\parindent 0pt\penalty10000
        \hskip 8 pt\nolinebreak\hbox{\ }\hfill\lower 8.5pt\null}
        \par\yyskip\pagebreak[1]}

\newtheorem{factfoo}[theoremfoo]{Fact}

\newenvironment{block}{\begin{list}{\hbox{}}{\leftmargin 1em
    \itemindent -1em \topsep 0pt \itemsep 0pt \partopsep 0pt}}{\end{list}}

\begin{document}

\lstset{language=Pascal}

\newcommand{\KN}{K_{\nat}}
\newcommand{\NRE}{\hbox{NUM-RED-EDGES\ }}
\newcommand{\NBE}{\hbox{NUM-BLUE-EDGES\ }}
\newcommand{\RED}{\hbox{RED\ }}
\newcommand{\BLUE}{\hbox{BLUE\ }}
\newcommand{\REDns}{\hbox{RED}}
\newcommand{\BLUEns}{\hbox{BLUE}}

\begin{center}
{\bf Proving Programs Terminate using}\\
{\bf Well-Founded Orderings, Ramsey's Theorem, and Matrices}
\end{center}

\centerline{by William Gasarch}

\begin{abstract}
Many programs allow the user to input data several times
during its execution. If the program runs forever the user
may input data infinitely often.
A program terminates if it terminates no matter what the user does.

We discuss various ways to prove that program terminates.
The proofs use well-founded orders, Ramsey Theorem, and matrices.
These techniques are used by real program checkers.
\end{abstract}

\noindent
{\bf General Terms:} Verification, Theory.

\noindent
{\bf Keywords and Phrase:} Proving programs terminate,
Well Orderings, Ramsey Theory, Matrices.

\section{Introduction}\label{se:intro}

We describe several ways to prove that programs terminate.
By this we mean terminate on {\it any} sequence of inputs.
The methods employed are well-founded orders, Ramsey's theorem, and matrices.
This paper is self contained; it does not require knowledge of any of these topics 
or of programming languages.
The methods we describe are used by real program checkers.

Our account is based on the articles of 
B. Cook, Podelski,Rybalchenko~\cite{CPRabs,CPRterm,proveterm,PRrank,ramseypl,PRtrans,DBLP:conf/tacas/PodelskiR11}
Lee, Jones, Ben-Amram~\cite{Lee:ranking,LJA}.
Termination checkers that use the methods discussed in this paper include:
\begin{enumerate}
\item
Loopfrog \url{http://www.verify.inf.unisi.ch/loopfrog/termination}.
\item
Terminator.
\url{http://www7.in.tum.de/~rybal/papers/}.
\item
ACL2 
\url{http://acl2s.ccs.neu.edu/acl2s/doc/}.
(Applicative common lisp 2).  
\item
AProVE 
\url{http://aprove.informatik.rwth-aachen.de/}.
(Automatic program verification environment).  
\item
Julia \url{http://julia.scienze.univr.it/}.
\end{enumerate}

\begin{convention}
The statement {\it The Program Terminates} means that it terminates
no matter what the user does. The user will be supplying inputs
as the program runs; hence we are saying that the user cannot
come up with some (perhaps malicious) inputs that make the program
run forever. A more realistic scenario is if the programs input is a sequence
of requests for devices. 
\end{convention}

In Section~\ref{se:notation} we establish a standard notation.
In Sections~\ref{se:simpleord},\ref{se:compordering} we prove particular programs
terminate using well founded orderings. In Section~\ref{se:wfgen} we present a general
theorem that encapsulates the technique of using well founded orderings.
In Section~\ref{se:useramsey} we prove a program terminates by using Ramsey Theory.
In Section~\ref{se:ramseygen} we prove a general theorem that encapsulates the
technique using Ramsey Theory.
In Sections~\ref{se:matrix},\ref{se:morematrix} we use Ramsey Theory and Matrices to prove
particular programs terminate, and also state a general theorem that encapsulates the technique
In Sections~\ref{se:useramsey2} and \ref{se:sub} we use Ramsey Theroy and invariants to
prove particular programs terminate.

All of the results are about showing particular types of prgrams can be be proven to
terminate. In Section~\ref{se:dec} we state (without proof) many theorems about particular
types of prgrams for which one can decide if the program terminates.

In Section~\ref{se:need} we discuss informally how much Ramsey Theory we need.
In particular, in most cases, the transitive Ramsey Theorem (which is a weaker version
of Ramsey Theory) suffices.

In the appendix we give some strange examples of programs and the proofs that they
terminate, and then give a tutorial on Ramsey's Theorem and the Transitive Ramsey Theorem.

\section{Notation and Definitions}\label{se:notation}

\begin{notation}~
\begin{enumerate}
\item
$\nat$ is the set $\{0,1,2,3,\ldots,\}$.
\item
$\Z$ is the set of integers, $\{\ldots, -2,-1,0,1,2,\ldots \}$.
\item
$\R$ is the set of reals.
\end{enumerate}
\end{notation}

\begin{notation}~
\begin{enumerate}
\item
In a program the command

\qquad $\x = \inp(\Z)$

\noindent
means that $\x$ gets an integer provided by the user.
\item
More generally, if $A$ is any set, then 

\qquad $\x = \inp(A)$

\noindent
means that $\x$ gets a value from A provided by the user.
\item
If we represent the set A by listing it out we will write
(for example)

\qquad $\x=\inp(\y,\y+2,\y+4,\y+6,\ldots)$

\noindent
rather than the proper but cumbersome

\qquad $x=\inp(\{y,y+2,y+4,y+6,\ldots\})$
\end{enumerate}

\noindent
In a program the command

\qquad $(\x,\y,\z) = (\inp(\Z),\inp(\nat),\inp(\nat))$

\noindent
means that $\x$ gets an integer provided by the user,
$\y$ gets a natural provied by the user, and
$\z$ gets a natural provied by the user.
One can generalize this to longer vectors of variables.

\noindent
In a program the command

\qquad $(\x,\y,\z) = (\y-\z,\x+\y+\z,\inp(\Z)$

\noindent
means that {\it simultaneously} $\x$ gets $\y-\z$,
$\y$ gets $\x+\y+\z$, and $\z$ gets an integer provided
by the user.
One can generalize this to longer vectors of variables and
any computable functions of them. 

\end{notation}

All of the programs we discuss do the following:
initially the variables get values supplied by the user,
then there is a \While loop. Within the \While loop the user can
specify which one of a set of statements get executed through the use of a variable called
{\it control}.
We focus on these programs for two reasons:
(1) programs of this type are a building block for more complicated programs,
and
(2) programs of this type will illustrate our points well.
One drawback is that the programs we present will not do anything of interest.

\begin{lstlisting}[frame=single,mathescape,float,title={Program~\simple}]
$(x,y,z) = (\inp(\nat),\inp(\nat),\inp(\nat))$
While $x>0$ and $y>0$ and $z>0$ 
	control = $\inp(1,2,3)$
	if control == 1
		$(x,y,z)=(x^2+10,y-x,z-10)$
	else
	if control == 2
		$(x,y,z)=(y^2+17,y-z^2,x-y)$
	else
	if control == 3
		$(x,y,z)=(y+17,xyz,x+y+z)$

\end{lstlisting}

\begin{example}~
\begin{enumerate}
\item
Program~\simple does not terminate since the user can set $(x,y,z)=(1,1,1)$ and
then keep setting control==3.
\item
Let $n,m\in\nat$.
Let $g_{i}$ as $1\le i\le m$ be  computable functions from $\Z^{n+1}$ to $\Z^n$.
These functions are used in Program~\gen which is very general.
 All of the programs in this paper will essentially be of this type.
\end{enumerate}
\end{example}

\bigskip
\bigskip
\bigskip

\begin{lstlisting}[frame=single,mathescape,title={Program~\gen}]
Comment: $X$ is $(x[1],\ldots,x[n])$
Comment: The $g_i$ are computable functions from $\Z^{n+1}$ to $\Z^n$
$X = (\inp(\Z),\inp(\Z),\ldots,\inp(\Z))$
While $x[1]>0$ and $x[2]>0$ and $\cdots$ and $x[n]>0$
 control = $\inp(1,2,3,...,m)$
  if control==$1$
   $X = g_{1}(X,\inp(\Z))$
  else
  if control==$2$
   $X = g_{2}(X,\inp(\Z))$
  else
   .
   .
   .
  else
  if control==$m$
   $X =g_{m}(X,\inp(\Z))$
\end{lstlisting}

We define this type of program formally.
We call it a {\it program} though it is actually a program of this
restricted type.
We also give intuitive comments in parenthesis.

\begin{definition}~
\begin{enumerate}
\item
A {\it program} is a tuple $(S,I,R)$ where the following hold.
\begin{itemize}
\item
$S$ is a decidable set of states. 
(If $(\x_1,\ldots,\x_n)$ are the variables in a program and they are of types 
$T_1,\ldots,T_n$ then $S=T_1\times\cdots\times T_n$.)
\item
$I$ is a decidable subset of $S$. ($I$ is the set of states that the program could be in initially.)
\item
$R\subseteq S\times S$ is a decidable set of ordered pairs. ($R(s,t)$ iff $s$ satisfies the condition of the \While loop and
there is some choice of instruction that takes $s$ to $t$.
Note that if $s$ does not satisfy the condition of the \While loop then there is no $t$ such that $R(s,t)$.
This models the \While loop termination condition.)
\end{itemize}
\item
A {\it computation} is a (finite or infinite) sequence of states $s_1,s_2,\ldots$
such that
\begin{itemize}
\item
$s_1\in I$.
\item
For all $i$ such that $s_i$ and $s_{i+1}$ exist, $R(s_i,s_{i+1})$.
\item
If the sequence is finite and ends in $s$ then there is no pair in $R$ whose first coordinate is $s$.
Such an $s$ is called {\it terminal}.
\end{itemize}
\item
A program {\it terminates} if every computation of it is finite.
\item
A {\it computational segment} is a sequence of states $s_1,s_2,\ldots,s_n$ such that, for all $1\le i\le n-1$,
$R(s_i,s_{i+1})$. Note that we do not insist that $s_1\in I$ nor do we insist that $s_n$ is a terminal state.
\end{enumerate}
\end{definition}

Consider Program~\examplexns.

\begin{lstlisting}[frame=single,mathescape,title=Program~\examplex]
$(x,y) = (\inp(\Z),\inp(\Z))$
While $x>0$ 
	control = $\inp(1,2)$
	if control == 1
		$(x,y)=(x+10,y-1)$
	else
	if control == 2
		$(x,y)=(y+17,x-2)$
\end{lstlisting}

Program~\examplex can be defined as follows:
\begin{itemize}
\item
$S=I=\Z\times\Z$.
\item
$R= \{ ((x,y),(x+10,y-1))\st x,y\ge 1\}\bigcup \{((x,y),(y+17,x-2))\st x,y\ge 1\}.$
\end{itemize}

\section{A Proof Using the Order $(\nat,\le)$ }\label{se:simpleord}

We show that every computation of Program~\wfN terminates.
To prove this we will find
a quantity that, during every iteration of the \While Loop, decreases.
None of $x,y,z$ qualify. However, the quantity $x+y+z$ does.
We use this in our proof.

\begin{lstlisting}[frame=single,mathescape,title={Program~\wfN}]
$(x,y,z) = (\inp(\Z), \inp(\Z), \inp(\Z))$
While $x>0$ and $y>0$ and $z>0$
	control = $\inp(1,2,3)$
	if control == 1 then
		$(x,y,z)=(x+1,y-1,z-1)$
	else
	if control == 2 then
		$(x,y,z)=(x-1,y+1,z-1)$
	else
	if control == 3 then
		$(x,y,z)=(x-1,y-1,z+1)$
\end{lstlisting}

\begin{theorem}\label{th:prog1}
Every computation of Program~\wfN is finite.
\end{theorem}

\begin{proof}

Let 
\begin{equation}
f(\x,\y,\z) = 
\begin{cases}
0 \text{ if   any of $x,y,z$ are $\le 0$;}
\\
\x + \y +\z \text{ otherwise.}
\cr
\end{cases}
\end{equation}

Assume, by way of contradition, that there is a nonterminating computation.

$$(\x_1,\y_1,\z_1), (\x_2,\y_2,\z_2),\ldots,$$

Before every iteration of the \While loop $f(\x,\y,\z)>0$.
After every iteration of the \While loop  $f(\x,\y,\z)$ has decreased.
Hence

$$f((\x_1,\y_1,\z_1)> f(\x_2,\y_2,\z_2)> \ldots,$$

This is impossible since the range of $f$ is $\nat$.

\end{proof}

The keys to the proof of Theorem~\ref{th:prog1} are (1) $x+y+z$ decreases with
every iteration, and (2) there is no infinite decreasing sequence of naturals.
We will later state a general theorem that can be used on any program
that satisfies generalizations of those properties.

\section{A Proof Using the Ordering $(\nat\times\nat\times\nat\times\nat,\llex)$}\label{se:compordering}

\begin{lstlisting}[frame=single,float,mathescape,title={Program~\wfNNN}]
$(w,x,y,z)= (\inp(\Z),\inp(\Z),\inp(\Z),\inp(\Z))$
While $w>0$ and $x>0$ and $y>0$ and $z>0$
	control = $\inp(1,2,3)$
	if control == 1 then
		$x=\inp(x+1,x+2,\ldots)$
		$w=w-1$
	else
	if control == 2 then
		$y=\inp(y+1,y+2,\ldots,)$
		$x=x-1$
	else
	if control == 3 then
		$z=\inp(z+1,z+2,\ldots)$
		$y=y-1$
\end{lstlisting}

To prove that 
every computation of Program~\wfNNN is finite we need to find
a quantity that, during every iteration of the \While Loop, decreases.
None of $x,y,z$ qualify. 
No arithmetic combination of $w,x,y,z$ qualifies.

\begin{definition}
Let $P$ be an order and $k\ge 1$.
The {\it lexicographic order}  on $P^k$ is the order
$$(a_1,\ldots,a_k) \llex (b_1,\ldots,b_k)$$
if for the least $i$ such that $a_i\ne b_i$, $a_i<b_i$.
\end{definition}

\begin{example}
In the order $(\nat^4,\llex)$
$$(1,10,10000000000,99999999999999) \llex (1,11,0,0).$$
\end{example}

We leave the following lemma to the reader.
\begin{lemma}\label{le:wf}
If $P$ is an well founded order and $k\ge 1$ then 
$(P,\llex)$ is a well founded order.
\end{lemma}

\begin{theorem}\label{th:prog2orderings}
Every computation of Program~\wfNNN is finite.
\end{theorem}

\begin{proof}

Assume, by way of contradiction, that there is a nonterminating computation.

$$(\w_1,\x_1,\y_1,\z_1), (\w_2,\x_2,\y_2,\z_2),\ldots,$$

Let

\begin{equation}
f(\w,\x,\y,\z) = 
\begin{cases}
(0,0,0,0) \text{ if any of $w,x,y,z$ are $\le$ 0;}
\\
(\w,\x,\y,\z) \text{ otherwise.}
\cr
\end{cases}
\end{equation}

We will be concerned with the order 
$(\nat^4,\llex)$.

\noindent
{\bf Claim 1:} In every iteration of the \While loop $f(\w,\x,\y,\z)$ decreases.

\noindent
{\bf Proof of Claim 1:}

Consider an iteration of the \While loop.
There are three cases.

\begin{enumerate}
\item
control=1: 
$w$ decreases by 1, 
$x$ increases by an unknown amount, 
$y$ stays the same,
$z$ stays the same.
Since the order is lexicographic, and $w$ is the first coordinate, 
the tuple decreases no matter how much $x$ increases.
\item
control=2: 
$w$ stays the same, 
$x$ decreases by 1,  
$y$ increases by an unknown amount,
$z$ stays the same.
Since the order is lexicographic, $w$ is the first coordinate and stays the same, and $x$ is the
second coordinate and decreases, the tuple decreases no matter how much $y$
 increases.
\item
control=3: $w$ stays the same, $x$ stays the same, $y$ decreases by 1, $z$ increases by an unknown amount.
This case is similar to the two other cases.
\end{enumerate}

\noindent
{\bf End of Proof of Claim 1}

Before every iteration of the \While loop $f(\w,\x,\y,\z)>0$.
After every iteration of the \While loop  $f(\w,\x,\y,\z)$ has decreased.
Hence

$$f(\w_1,\x_1,\y_1,\z_1)> f(\w_1,\x_2,\y_2,\z_2)> \ldots,$$

This is impossible since the range of $f$ if $P$ and, by Lemma~\ref{le:wf},
$P$ has no infinite descending sequences.

\end{proof}

\section{A General Theorem about Proving Programs Terminate Using Well Founded Orderings}\label{se:wfgen}

The proofs of Theorems~\ref{th:prog1}
and~\ref{th:prog2orderings} look very much alike.
There is a general theorem, due to Floyd~\cite{floydpl},
that captures both of these proofs
and many more.

\begin{definition}
An order $T$ is {\it well-founded} if every nonempty subset has a minimal element.
Note that if $T$ is well-founded then there are no infinite descending sequences of elements of $T$.
\end{definition}

\begin{theorem}\label{th:useorderings}
Let $PROG=(S,I,R)$ be a program.
Assume that 
there is a well-founded order $(P,<_P)$, and a map $f:S\into P$ such that
if $R(s,t)$ then $f(t) <_P f(s).$
Then any computation of $PROG$ is finite.
\end{theorem}

\begin{proof}
Assume the premise holds. We denote $<_P$ by $<$.
Assume, by way of contradiction, that the program does not terminate.
Then there exists an infinite sequence of states

$$s_1,s_2,s_3,\ldots,$$

such that, for all $i$, $R(s_i,s_{i+1})$.  By the premise on $f$ we have

$$f(s_1) > f(s_2) > f(s_3) > \cdots $$

This contradicts $<$ being a well-founded order.

\end{proof}

\begin{note}
It turns out that this theorem is iff. That is, if every computation of $PROG$ is finite
then there is a (perhaps contrived) well-order that satisfies the premise.
\end{note}

\section{A Proof Using Ramsey's Theorem}\label{se:useramsey}

In the proof of Theorem~\ref{th:prog2orderings} we showed that during every single step
of Program~\wfNNN the quantity $(w,x,y,z)$ decreased with respect to the order $\llex$. 
The proof of termination was easy in that we only
had to deal with one step but hard in that we had to deal with the lexicographic order
on $\nat\times\nat\times\nat\times\nat$ rather than just the order $\nat$.

In this section we will prove that Program~\wfNNN terminates in a different way.
We will not need an order on 4-tuples. We will only deal with
$w,x,y,z$ individually. However, we will need to prove that, for
{\it each} finite \compsegns, at least one of $w,x,y,z$ decreases.

We will use the infinite Ramsey's Theorem. In the Appendix we will give some history and the proof of Ramsey's Theorem.
For now we state it and use it.

\begin{notation}~
\begin{enumerate}
\item
If $n\ge 1$ then $K_n$ is the complete graph with vertex set $V=\{1,\ldots,n\}$.
\item
$K_\nat$ is the complete graph with vertex set $\nat$.
\end{enumerate}
\end{notation}

\begin{definition}\label{de:homog}
Let $c,n\ge 1$.
Let $G$ be $K_n$ or $K_\nat$. Let $COL$ be a $c$-coloring of the edges of $G$.
A set of vertices $V$ is {\it homogeneous with respect to $COL$} if
all the edges between vertices in $V$ are the same color.
We will drop the {\it with respect to $COL$} if the coloring is understood.
\end{definition}

\noindent
{\bf Infinite Ramsey's Theorem:}

\begin{theorem}\label{th:ramsey}
Let $c\ge 1$.
For every $c$-coloring of the the edges of $K_\nat$ there exists  an infinite  homogeneous set.
\end{theorem}

\begin{theorem}\label{th:prog2ramsey}
Every computation of Program~\wfNNN is finite.
\end{theorem}

\begin{proof}

We show Program~\wfNNN terminates. Assume, by way of contradiction,
that there is an infinite computation.
Let this computation be

\smallskip

$$(\w_1,\x_1,\y_1,\z_1), (\w_2,\x_2,\y_2,\z_2), \ldots.$$

\smallskip

We show that for each finite \compseg one of $w,x,y$ will decrease.
Let $i<j$. We look at the finite \compseg

$$(\w_i,\x_i,\y_i,\z_i), (\w_{i+1},\x_{i+1},\y_{i+1},\z_{i+1}), \ldots, (w_j,x_j,y_j,z_j).$$

There are several cases.

\begin{enumerate}
\item
If control=1 ever occurs in the segment then $w_i>w_j$.
No other case makes $w$ increase, so we are done.
In all later cases we can assume that control is never 1 in the segment.
\item
If control=2 ever occurs in the segment then $x_i>x_j$. Since control=1 never occurs and
control=3 does not make $x$ increase, $x$ decreases, and we are done.
In all later cases we can assume that control is never 1 or 2 in the segment.
\item
If control=3 is the only case that occurs in the segment then $y_i>y_j$.
\end{enumerate}

Since in for each finite \compseg one of $w,x,y$ decreases we have that,
for all $i<j$, either $\w_i>\w_j$ or $\x_i>\x_j$ or $\y_i>\y_j$.
We use this to create a coloring of the edges of $K_\nat$.
Our colors are $W,X,Y$. In the coloring below each case assumes that
the cases above it did not occur.

\begin{equation}
COL(i,j) = 
\begin{cases}
W \text{ if $\w_i>\w_j$;}
\\
X \text{ if $\x_i>\x_j$;}
\\
Y \text{ if $\y_i>\y_j$.}
\cr
\end{cases}
\end{equation}

By Ramsey's Theorem there is an infinite set 

$$i_1 < i_2 < i_3 < \cdots $$

such that

$$COL(i_1,i_2) = COL(i_2,i_3) = \cdots.$$

(We actually know more. We know that {\it all} pairs  $(i_j,i_k)$
have the same color. We do not need this fact here; however,
see the second note after Theorem~\ref{th:useramseywf}.)

Assume the color is $W$ (the cases for $X,Y$ are similar).
Then

$$\w_{i_1} > \w_{i_2} > \w_{i_3} > \cdots. $$

Hence eventually $\w$ must be less than 0. When this happens the program terminates.
This contradicts the program not terminating.
\end{proof}

\section{A General Theorem about Proving Programs Terminate Using Ramsey Theorem}\label{se:ramseygen}

The keys to the proof of Theorem~\ref{th:prog2ramsey} are (1) in every finite \compseg one of $w,x,y$ decreases,
and (2) by Ramsey's Theorem any nonterminating computation leads to an infinite decreasing
sequence in a well-founded set.
These ideas are from Theorem 1 of~\cite{ramseypl}, though similar ideas were in~\cite{LJA}.

Theorem 1 of~\cite{ramseypl} is a very general statement about program termination.
We present three theorems in increasing order of generality. The last one  is Theorem 1 of~\cite{ramseypl}.

\begin{theorem}\label{th:useramsey}
Let $PROG=(S,I,R)$ be a program of the form of Program~\genns.
Note that the variables are $x[1],\ldots,x[n]$.
Assume that for each  \compseg $t_1,\ldots,t_L$ there exists a $1\le k\le m$ such that
$x[k]$ in $t_1$ is strictly less than $x[k]$ in $t_L$.
Then any computation of $PROG$ is finite.
\end{theorem}

\begin{proof}

We show Program $(S,I,R)$ terminates. Assume, by way of contradiction,
that there is an infinite computation.
Let this computation be

\smallskip

$$s_1,s_2,s_3,\ldots $$
where each $s_i$ is an $n$-tuple of values for $(x[1],\ldots,x[n])$.

By the premise, for every $i<j$, in the finite \compseg
$$s_i,s_{i+1},\ldots,s_j$$
there is a $k$ such that $x[k]$ in $s_i$ is less than $x[k]$ in $s_j$.

We use this to create a coloring of the edges of $K_\nat$.
Our colors are $\{1,\ldots,m\}$. 
$COL(i,j)$ is the least index $k$ such that $x[k]$ in $s_i$ is greater than
$x[k]$ in $s_j$.

By Ramsey's Theorem there is an infinite set 
$$i_1 < i_2 < i_3 < \cdots $$
and a color $L$ such that

$$L=COL(i_1,i_2)=COL(i_2,i_3)=\cdots.$$

Hence the value of $x[L]$ in $s_1$ is larger than it is in $s_2$ is larger than it is in $s_3$, etc.
This means that there is a time when the value of $x[L]$ is $\le 0$. Hence the program terminates.
This is a contradiction.
\end{proof}

To prove that a program terminates we might use some
function of the variables rather than the variables themselves.
The next theorem, which is a generalization of Theorem~\ref{th:useramsey},
captures this.

\begin{theorem}\label{th:useramseyf}
Let $PROG=(S,I,R)$ be a program of the form of Program~\genns.
Note that the variables are $x[1],\ldots,x[n]$. We denote the vector of variables by $\vec x$.
Assume there exists functions $f_1(\vec x),\ldots,f_M(\vec x)$ with range $\nat$ such that the
following holds:
For each \compseg $t_1,\ldots,t_L$ there exists a $1\le k\le M$ such that
$f_k(\vec x)$ in $t_1$ is strictly less than $f_k(\vec x)$ in $t_L$.
Then any computation of $PROG$ is finite.
\end{theorem}

\begin{sketch}

This proof is virtually identical to the proof of Theorem~\ref{th:useramsey}
The only difference comes towards the end, so we do the last few lines.

$COL(i,j)$ is the least index $k$ such that $f(\vec x)$ in $s_i$ is greater than
$f(\vec x)$ in $s_j$.

By Ramsey's Theorem there is an infinite set 

$$i_1 < i_2 < i_3 < \cdots $$

and a color $L$ such that

$$L=COL(i_1,i_2)=COL(i_2,i_3)=\cdots.$$

Hence the value of $f_L(\vec x)$ in $s_1$ is larger than it is in $s_2$ is larger than it is in $s_3$, etc.
This means that there is a time when the value of $f_L(\vec x)$ is $\le -1$. This is a contradiction since
$f$ has range $\nat$.
\end{sketch}

In the statement of Theorem~\ref{th:useramseyf} the functions $f$ mapped to the natural numbers.
What was it about the natural numbers that we used? At first glance it seems like we only use
that $-1\notin\nat$. However, we really used that $\nat$ is well-founded. This leads to a more
general theorem.

\begin{theorem}\label{th:useramseywf}
Let $PROG=(S,I,R)$ be a program of the form of Program~\genns.
Note that the variables are $x[1],\ldots,x[n]$. We denote the vector of variables by $\vec x$.
Assume there exists functions $f_1(\vec x),\ldots,f_M(\vec x)$ 
such that $f_i$ has range $P_i$ where $P_i$ is a well-founded set.
for each \compseg $t_1,\ldots,t_n$ there exists a $1\le k\le M$ such that
$f_k(\vec x)$ in $t_1$ is strictly less than (using the order $P_k$) $f_k(\vec x)$ in $t_n$.
Then any computation of $PROG$ is finite.
\end{theorem}

\begin{sketch}

This proof is virtually identical to the proof of Theorem~\ref{th:useramsey}
The only difference comes towards the end, so we do the last few lines.

$COL(i,j)$ is the least index $k$ such that $f(\vec x)$ in $s_i$ is greater than (using the order $P_k$)
$f(\vec x)$ in $s_j$.

By Ramsey's Theorem there is an infinite set 

$$i_1 < i_2 < i_3 < \cdots $$

and a color $L$ such that

$$L=COL(i_1,i_2)=COL(i_2,i_3)=\cdots.$$

Hence the value of $f_L(\vec x)$ in $s_1$ is larger (using the order $P_k$) 
than it is in $s_2$ is larger than (using the order $P_k$) it is in $s_3$, etc.
Hence we have an infinite decreasing sequence in $P_k$. 
This is a contradiction since $P_k$ is a well-founded ordering.
\end{sketch}

\begin{note}
It turns out that this theorem is iff. That is, if every omputation of $PROG$ is finite
then there are (perhaps contrived) functions $f_i$ and well-founded orderings $P_i$ as
stated in Theorem~\ref{th:useramseywf}.
This is the actual statement of Theorem 1 of~\cite{ramseypl}.
\end{note}

\begin{note}
The proofs of Theorems~\ref{th:prog2ramsey}, \ref{th:useramsey} and \ref{th:useramseywf}
do not need the full strength of Ramsey's Theorem. 
Consider Theorem~\ref{th:useramsey}.
For any $i,j,k$ if
$COL(i,j)=a$ (so $a$ is the least number such that $x[a]$ in $s_i$ is greater than $x[a]$ in $s_j$)
$COL(j,k)=a$ (so $a$ is the least number such that $x[a]$ in $s_j$ is greater than $x[a]$ in $s_k$)
one can show $COL(i,k)=a$.
Such colorings are called {\it transitive}.
Hence we only need Ramsey's Theorem for transitive colorings.
We discuss this further in Section~\ref{se:need}.
\end{note}

\section{A Proof Using Matrices and Ramsey's Theorem}\label{se:matrix}

Part of the proof of Theorem~\ref{th:prog2ramsey}
involved showing that, for any finite \compseg of Program~\wfNNNns, one of $w,x,y,z$ decreases.
Can such proofs be automated?

Ben-Amram~\cite{BA:delta} developed a way to partially automate
such proofs. He uses matrices and Ramsey's Theorem.
An earlier version by Lee, Jones, and Ben-Aram~\cite{LJA} used size-change graphs 
instead of matrices. We discuss the difference later.

We use Ben-Amram's matrix techniques to give a proof that Program~\wfNNN terminates.
We will then discuss their general technique.

Program~\wfNNN has variables $w,x,y,z$. To use Theorem~\ref{th:useramsey} on it we need
to know that in every finite \compseg one of these variables decreases.
We would rather reason about what happens during one step. Let us capture what we do
know about one step.

If control=1 then 
\[
\begin{array}{rl}
w= & w-1\cr
x= & \inp(x+1,x+2,\ldots)\cr
y= & y\cr
z= & z\cr
\end{array}
\]

We represent this by a matrix. The rows and columns are both indexed by the variables,
so it will be a four by four matrix. In the (say) $(w,y)$ entry we put the difference
between the new $y$ and the old $w$. If we do not know the difference we put $\infty$
(this will happen most of the time).
It is easy to see that the matrix is:

\[ C_1 = \left ( \begin{array}{cccc}
 -1         & \infinity & \infinity  & \infinity \cr
  \infinity & \infinity & \infinity  & \infinity \cr
  \infinity & \infinity & 0          & \infinity \cr
  \infinity & \infinity & \infinity & 0 \cr
\end{array} \right )
\]

The matrix for control=2 is

\[ C_2= \left ( \begin{array}{cccc}
  0         & \infinity & \infinity  & \infinity \cr
  \infinity & -1        & \infinity  & \infinity \cr
  \infinity & \infinity & \infinity  & \infinity \cr
  \infinity & \infinity & \infinity & 0 \cr
\end{array} \right )
\]

The matrix for control=3 is

\[C_3= \left ( \begin{array}{cccc}
  0         & \infinity & \infinity  & \infinity \cr
  \infinity & 0         & \infinity  & \infinity \cr
  \infinity & \infinity & -1         & \infinity \cr
  \infinity & \infinity & \infinity & \infinity \cr
\end{array} \right )
\]

Clearly if the program executes any one of these commands then
some variable decreases. In terms of the matrices this means that
some entry on the diagonal is negative.

We need that any finite sequence of instructions leads to some variable decreacing 
We want to express any finite sequence of instructions as a matrix. How? 

\begin{definition}
If $A$ and $B$ are $n\times n$ matrices then we define (just for this paper)
the product $AB$ in the following (nonstandard) way:

$$AB[i,j] = \min_{1\le k\le n} \{ a_{ik} + b_{kj} \}.$$

By convention, for any $x\in \nat \cup \{\infinity\}$,
$\infinity + x = x + \infinity = \infinity$.
\end{definition}

We leave the proof of the following easy lemma to the reader.

\begin{lemma}\label{le:matrix}
Let $\vec x$ be variables and $g_1(\vec x)$, $g_2(\vec x)$ be computable functions.
Let $PROG_1$ be the short program $\vec x = g_1(\vec x)$.
Let $PROG_2$ be the short program $\vec x = g_2(\vec x)$.
(We think of $PROG_1$ and $PROG_2$ as being what happens in the various control cases.)
Let $C_1$ be the matrix that represents what is known whenever $PROG_1$ is executed.
Let $C_2$ be the matrix that represents what is known whenever $PROG_2$ is executed.
Then the matrix product $C_1C_2$ as defined above represents what is known when
$PROG_1$ and then $PROG_2$ are executed.
\end{lemma}

Hence every finite sequence of instructions corresponds to some finite product
of $C_1$'s, $C_2$'s and $C_3$'s.  In the case at hand we need only show that
every such product has a negative number on some diagonal. We state this in general.

\begin{theorem}\label{th:matrix}
Let $PROG=(S,I,R)$ be a program in the form of Program~\genns.
Let $C_1, C_2,\ldots, C_m$
be the matrices associated to control=1, $\ldots$, control=m cases.
If every product of the $C_i$'s yields a matrix with a negative integer on the
diagonal then the program terminates.
\end{theorem}

\begin{proof}
Consider \compseg $s_1,\ldots,s_n$.
Let the corresponding matrices be $C_{i_1},\ldots,C_{i_n}$.
By the premise the product of these matrices has a negative integer on the diagonal.
Hence some variable decreases.
By Theorem~\ref{th:useramsey} the program terminates.
\end{proof}

\begin{note}
Lee, Jones, and Ben-Amram used size-change graphs rather than matrices.
Their results can be interpreted as matrices where, instead of having the
difference, you have whether or not the (say) old $y$ is bigger than
the old $x$, or smaller, or unknown.
\cite{LJA}

\end{note}

In the case at hand it may seem difficult to show that {\it every}
product $C_1$'s, $C_2$'s and $C_3$'s has a negative number on the diagonal.
Howver, we can show this:

\begin{theorem}\label{th:prog2matrix}
Every computation of Program~\wfNNN is finite.
\end{theorem}

\begin{proof}

Let $C_1,C_2,C_3$ be the matrices that represent the cases
of Control=1,2,3 in Program~\wfNNNns. (These matrices are above.)
We show that the premise of Theorem~\ref{th:matrix} holds.
To do this we prove items 0-7 below.
Item 0 is easily proven directly.
Items 1,2,3,4,5,6,7 are easily proven by induction
on the number of matrices being multiplied.

\begin{enumerate}

\item[0.]
$C_1C_2=C_2C_1$, $C_1C_3=C_3C_1$, $C_2C_3=C_3C_2$.

\item[1.]
For all $a\ge 1$
\[ C_1^a = \left ( \begin{array}{cccc}
  -a  & \infinity & \infinity  & \infinity \cr
  \infinity & \infinity & \infinity  & \infinity \cr
  \infinity & \infinity & 0          & \infinity \cr
  \infinity & \infinity & \infinity  & 0         \cr
\end{array} \right )
\]

\item[2.]
For all $b\ge 1$ 
\[C_2^b =  \left ( \begin{array}{cccc}
  0   & \infinity & \infinity  & \infinity \cr
  \infinity & -b        & \infinity  & \infinity \cr
  \infinity & \infinity & \infinity  & \infinity \cr
  \infinity & \infinity & \infinity  & 0         \cr
\end{array} \right )
\]

\item[3.]
For all $c\ge 1$
\[C_3^c =  \left ( \begin{array}{cccc}
  0   & \infinity & \infinity  & \infinity \cr
  \infinity & 0         & \infinity  & \infinity \cr
  \infinity & \infinity & -c         & \infinity \cr
  \infinity & \infinity & \infinity  & \infinity \cr
\end{array} \right )
\]

\item[4.]
For all $a,b\ge 1$
\[C_1^aC_2^b =  \left ( \begin{array}{cccc}
  -a  & \infinity & \infinity  & \infinity \cr
  \infinity & -b        & \infinity  & \infinity \cr
  \infinity & \infinity & 0          & \infinity \cr
  \infinity & \infinity & \infinity  & 0         \cr
\end{array} \right )
\]

\item[5.]
For all $a,c\ge 1$
\[ C_1^aC_3^c = \left ( \begin{array}{cccc}
  -a  & \infinity & \infinity  & \infinity \cr
  \infinity & 0         & \infinity  & \infinity \cr
  \infinity & \infinity & -c         & \infinity \cr
  \infinity & \infinity & \infinity  & \infinity \cr
\end{array} \right )
\]

\item[6.]
For all $b,c\ge 1$
\[ C_2^bC_3^c = \left ( \begin{array}{cccc}
  0   & \infinity & \infinity  & \infinity \cr
  \infinity & -b        & \infinity  & \infinity \cr
  \infinity & \infinity & \infinity  & \infinity \cr
  \infinity & \infinity & \infinity  & \infinity \cr
\end{array} \right )
\]

\item[7.]
For $a,b,c\ge 1$
\[ C_1^aC_2^bC_3^c = \left ( \begin{array}{cccc}
  -a  & \infinity & \infinity  & \infinity \cr
  \infinity & \infinity & \infinity  & \infinity \cr
  \infinity & \infinity & \infinity  & \infinity \cr
  \infinity & \infinity & \infinity  & 0         \cr
\end{array} \right )
\]
\end{enumerate}

Since the multiplication
of these matrices is commutative we need only concern ourselves with $C_1^aC_2^bC_3^c$
for $a,b,c \in \nat$.
In all of the cases below $a,b,c\ge 1$.

\begin{enumerate}
\item
$C_1^a$: $w$ decreases.
\item
$C_2^b$: $x$ decreases.
\item
$C_3^c$: $y$ decreases.
\item
$C_1^aC_2^b$:  Both $w$ and $x$ decrease.
\item
$C_1^aC_3^c$: Both $w$ and $y$ decrease.
\item
$C_2^bC_3^c$: $x$ decreases.
\item
$C_1^aC_2^bC_3^c$: $w$ decreases.
\end{enumerate}

\end{proof}

The keys to the proof of Theorem~\ref{th:prog2matrix} are
(1) represent how the old and new variables relate after one iteration with a matrix,
(2) use these matrices and a type of matrix multiplication to determine that
for every finite \compseg some variable decreases,
(3) use Theorem~\ref{th:useramsey} to conclude the program terminates.

Theorem~\ref{th:matrix} leads to the following algorithm
to test if a  programs terminates.
There is one step (alas, the important one) which we do not say how to do.
If done in the obvious way it may not halt.

\begin{enumerate}
\item
Input Program P.
\item
Form matrices for all the cases of control.
Let them be $C_1,\ldots,C_m$.
\item
Find a finite set of types of matrices ${\cal M}$ such that
that any product of the $C_i$'s (allowing repeats) is in ${\cal M}$.
(If this step is implemented by looking at all possible products
until a pattern emerges then this step might not terminate.)
\item
If all of the elements of ${\cal M}$ have some negative diagonal element
then output {\it YES the program terminates!}
\item
If not the then output {\it I DO NOT KNOW if the program terminates!}
\end{enumerate}

If all products of matrices fit a certain pattern, as they did in the proof of Theorem~\ref{th:prog2matrix},
then this idea for an algorithm will terminate. Even in that case, it may output {\it I DON"T KNOW if the program terminates!}.
However, this algorithm can be used to prove that some programs terminate, just not all.
It cannot be used to prove that a program will not terminate.

The premise of Theorem~\ref{th:matrix} is designed so that we can apply Theorem~\ref{th:useramsey}.
Hence we are only looking at the variables of the program and the natural numbers.
We generalize Theorem~\ref{th:matrix} so it feeds into Theorem~\ref{th:useramseyf},
We omit the proof which is similar to that of Theorem~\ref{th:matrix}.

Let $PROG=(S,I,R)$ be a program in the form of Program~\genns.
Note that the variables are $x[1],\ldots,x[n]$. We denote the vector of variables by $\vec x$.
Let functions $f_1(\vec x),\ldots,f_M(\vec x)$ have range $\nat$.
We can now form $k\times k$ matrices $C_1,\ldots,C_m$ such that
matrix $C_L[i,j]$ is the difference between the new $f_j(\vec x)$ and the old $f_i(\vec x)$.

\begin{theorem}\label{th:matrixf}
Let $PROG=(S,I,R)$ be a program in the form of Program~\genns.
Note that the variables are $x[1],\ldots,x[n]$. We denote the vector of variables by $\vec x$.
Let functions $f_1(\vec x),\ldots,f_M(\vec x)$ have range $\nat$.
Assume that $C_1,\ldots,C_m$ are the matrices associated to them as noted above.
If every product of the matrices has a negative number on the diagonal then
the program terminates.
\end{theorem}

Is there a further generalization of Theorem~\ref{th:matrix} that feeds into
Theorem~\ref{th:useramseywf}. Recall in the premise of Theorem~\ref{th:useramseywf} 
the functions $f$ has range some well-founded order. The matrices we work with
deal with differences. Since the different $f$'s in Theorem~\ref{th:useramseywf} have 
ranges in different
well-founded orders, we cannot take their difference. What if we require that the
$f$'s all have the same well-founded order as their range? This still does not work
since some well-founded order (e.g., $(\nat\times\nat,\llex)$) do not have a notion of
difference. The approach of Lee, Jones, and Ben-Amram that used size-change graphs
insead of matrices (see note after Theorem~\ref{th:matrix}) might work here.

\section{Another Proof Using Matrices and Ramsey's Theorem}\label{se:morematrix}

\begin{lstlisting}[frame=single,float,mathescape,title={Program~\ramsey}]
$(x,y) = (\inp(\Z),\inp(\Z))$
While $x>0$ and $y>0$
	control = $\inp(1,2)$
	if control == 1 then
		$(x,y)=(x-1,x)$
	else
	if control == 2 then
		$(x,y)=(y-2,x+1)$
\end{lstlisting}

We prove Program~\ramsey terminates using matrices. 
The case control=1  is represented by the matrix
$$ C_1 = \mat{ -1 & 0 \\ \infty & \infty }. $$
The case control=2  is represented by the matrix
$$ C_2 = \mat{ \infty & -2 \\  1 & \infty }. $$
This will not work! Note that $C_2$ is has no negative numbers
on its diagonal. Hence we cannot use these matrices in our proof!
What will we do!?
Instead of using $x,y$ we will use $x,y$, and $x+y$.
We comment on whether or not you can somehow use $C_1$ and $C_2$ after the proof.

\begin{theorem}\label{th:prog3matrix}
Every computation of Program~\ramsey is finite.
\end{theorem}

\begin{proof}
We will use Theorem~\ref{th:matrixf} with
functions $x,y$, and $x+y$.
Note that $x+y$ is not one of the original variables which
is why we need Theorem~\ref{th:matrixf} rather than Theorem~\ref{th:matrix}.

The control=1 case of Program~\ramsey corresponds to

\[ D_1 = \left ( \begin{array}{ccc}
  -1        & 0         & 1  \cr
  \infinity & \infinity & \infinity  \cr
  \infinity & \infinity & \infinity \cr
\end{array} \right )
\]

The control=2 case of Program~\ramsey corresponds to 

\[ D_2 = \left ( \begin{array}{ccc}
  \infinity & 1         & \infinity  \cr
  -2        & \infinity & \infinity  \cr
  \infinity & \infinity & -1        \cr
\end{array} \right )
\]

We show that the premises of Theorem~\ref{th:matrixf} hold.
The following are true and easily proven by induction
on the number of matrices being multiplied.

\begin{enumerate}

\item
For all $a\ge 1$ 

\[ D_1^a = \left ( \begin{array}{ccc}
  -a        & -a+1      & -a+2 \cr
  \infinity & \infinity & \infinity  \cr
  \infinity & \infinity & \infinity \cr
\end{array} \right )
\]

\item
For all $b\ge 1$, $b$ odd, $b=2d-1$,

\[ D_2^b = \left ( \begin{array}{ccc}
  -d        & \infinity & \infinity \cr
  \infinity & -d        & \infinity  \cr
  \infinity & \infinity & -2d       \cr
\end{array} \right )
\]

\item
For all $b\ge 2$, $b$ even, $b=2e$, 

\[ D_2^b = \left ( \begin{array}{ccc}
  \infinity & -e+1      & \infinity \cr
  -e-2      & \infinity & \infinity  \cr
  \infinity & \infinity & -2e-1     \cr
\end{array} \right )
\]

\item
For all $a,b\ge 1$, $b$ odd, $b=2d-1$.

\[ D_1^aD_2^b = \left ( \begin{array}{ccc}
  -a-d & -a-d+1  & -a-2d+2 \cr
  \infinity & \infinity & \infinity  \cr
  \infinity & \infinity & \infinity  \cr
\end{array} \right )
\]

\item
For all $a,b\ge 1$, $b$ even, $b=2e$.

\[ D_1^aD_2^b = \left ( \begin{array}{ccc}
  -a-e-1 & -a-e+1  & -a-2e+1 \cr
  \infinity & \infinity & \infinity  \cr
  \infinity & \infinity & \infinity  \cr
\end{array} \right )
\]

\item
For all $a,b\ge 1$, $a$ is odd,

\[ D_2^aD_1^b = \left ( \begin{array}{ccc}
  \infinity & \infinity & \infinity  \cr
  -(\floor{a/2}+b+2 & -(\floor{a/2}+b+1 & -(\floor{a/2}+b \cr
  \infinity & \infinity & \infinity  \cr
\end{array} \right )
\]

\item
If $a,b\ge 1$, $a$ is even,

\[ D_2^aD_1^b = \left ( \begin{array}{ccc}
  -(a/2)+b & -(a/2)+b-1 & -\floor{a/2}+b-2 \cr
  \infinity & \infinity & \infinity  \cr
  \infinity & \infinity & \infinity  \cr
\end{array} \right )
\]
\end{enumerate}

We use this information to formulate a lemma.

\noindent
{\bf Convention:}
If we put $< 0$ $(\le 0$) in an entry of a matrix it means that the entry is some
integer less than 0 (less than or equal to 0).
We might not know what it is.

\noindent
{\bf Claim:}
For all $n\ge 2$, any product of $n$ matrices all of which are $D_1$'s and $D_2$'s
must be of one of the following type:
\begin{enumerate}

\item
\[ \left ( \begin{array}{ccc}
  <0        & \le 0     &\le 0              \cr
  \infinity & \infinity & \infinity  \cr
  \infinity & \infinity & \infinity  \cr
\end{array} \right )
\]

\item
\[ \left ( \begin{array}{ccc}
  \infinity & \infinity & \infinity  \cr
  < 0       & < 0       & < 0         \cr
  \infinity & \infinity & \infinity  \cr
\end{array} \right )
\]

\item
\[ \left ( \begin{array}{ccc}
  <0        & \infinity & \infinity  \cr
  \infinity & < 0       & \infinity   \cr
  \infinity & \infinity & < 0        \cr
\end{array} \right )
\]

\item
\[ \left ( \begin{array}{ccc}
  \infinity & < 0       & \infinity  \cr
  < 0       & \infinity & \infinity   \cr
  \infinity & \infinity & < 0        \cr
\end{array} \right )
\]

\end{enumerate}

\noindent
{\bf End of Claim}

This can be proved easily by induction on $n$.
\end{proof}

One can show that every computation of Program~\ramsey terminates
using the original matrices $2\times 2$ matrices $C_1, C_2$.
Ben-Amram has done this and has allowed us to
place his proof in the appendix of this paper.

\section{A Proof Using Transition Invariants and Ramsey's Theorem}\label{se:useramsey2}

We present an example from~\cite{ramseypl} of a program (Program~\ramsey)
where the proof of termination using Ramsey's Theorem is obtained by using transition invariants
(to be defined).
Podelski and Rybalchenko found this proof by hand and later their
termination checker found it automatically.
A proof of termination using a well-founded order seems difficult to find.
Ben-Amram and Lee~\cite{BA:mcs,Lee:ranking} have shown that a termination proof that explicitly
exhibits a well-founded order can be automatically derived when the
matrices  only use entries $0,-1$, and $\infty$.
Alas, Program~\ramsey is not of this type; however, using some manipulation Ben-Amram (unpublished) has
used this result to show that Program~\ramsey terminates.  
(The proof is in the Appendix.)
Hence there is a proof that Program~\ramsey terminates that uses a well-founded order; however, it was difficult to obtain.

\begin{theorem}\label{th:prog3ramsey}
Every computation of Program~\ramsey is finite.
\end{theorem}

\begin{proof}

We assume that the \compseg enters the \While loop, else
the program has already terminated.

We could try to show that, in each finite \compsegns, either $x$ or $y$ decreases.
This statement is true but seems hard to prove directly.
Instead we show that either $x$ or $y$ or $x+y$ decreases. This turns out
to be easier. Intuitively we are loading our induction hypothesis.
We now proceed formally.

We show that the premises of Theorem~\ref{th:useramseyf} hold with 
$f_1(\x,\y)=\x$, $f_2(\x,\y)=\y$, and $f_3(\x,\y)=\x+\y$.
It may seem as if knowing that $x+y$ decreases you know that either $x$ or $y$ decreases.
However, in our proof, we will {\it not} know which of $x,y$ decreases.
Hence we must use $x,y$, and $x+y$.

\noindent
{\bf Claim 1:} For each finite \compsegns, one of $x,y,x+y$ decreases.

\noindent
{\bf Proof of Claim 1:}

We want to prove that, for all $n\ge 2$, for each \compsegs of length $n$
$$(\x_1,\y_1), (\x_2,\y_2),\ldots,(\x_n,\y_n),$$
either $\x_1>\x_n$ or $\y_1>\y_n$ or $\x_1+\y_1 > \x_n + \y_n$.
However, we will prove something stronger.
We will prove that, for all $n\ge 2$, for each \compsegs of length $n$
$$(\x_1,\y_1), (\x_2,\y_2),\ldots,(\x_n,\y_n),$$
one of the following occurs.

\begin{enumerate}
\item[(1)]
$\x_1>0$ and $\y_1>0$ and $\x_n<\x_1$ and $\y_n \le \x_1$ (so $x$ decreases),
\item[(2)]
$\x_1>0$ and $\y_1>0$ and $\x_n<\y_1-1$ and $\y_n\le \x_1+1$ (so $x+y$ decreases),
\item[(3)]
$\x_1>0$ and $\y_1>0$ and $\x_n<\y_1-1$ and $\y_n < \y_1$ (so $y$ decreases),
\item[(4)]
$\x_1>0$ and $\y_1>0$ and $\x_n<\x_1$ and $\y_n < \y_1$ (so $x$ and $y$ both decreases, though we just need one of them).
\end{enumerate}

(We will later refer to the OR of these four statements as {\it the invariant}.)

We prove this by induction on $n$.

\noindent
{\bf Base Case:} $n=2$ so we only look at one instruction.

If $(\x_{2},\y_{2})=(\x_1-1,\x_1)$ is executed then (1) holds.

If $(\x_{2},\y_{2})=(\y_1-2,\x_1+1)$ is executed then (2) holds.

\bigskip

\noindent
{\bf Induction Step}: We prove Claim 1 for $n+1$ assuming it for $n$.
There are four cases, each with two subcases.

\begin{enumerate}
\item
$\x_n<\x_1$ and $\y_n \le \x_1$.

\begin{enumerate}
\item
If $(\x_{n+1},\y_{n+1})=(\x_n-1,\x_n)$ is executed then

\begin{itemize}
\item
$\x_{n+1} = \x_n-1 < \x_1 -1 < \x_1$
\item
$\y_{n+1} = \x_n < \x_1$
\end{itemize}

Hence (1) holds.

\item
If $(\x_{n+1},\y_{n+1})=(\y_n-2,\x_n+1)$ is executed then 

\begin{itemize}
\item
$\x_{n+1}=\y_n-2 \le \x_1-2 < x_1$
\item
$\y_{n+1} = \x_n + 1 \le \x_1$
\end{itemize}

Hence (1) holds.
\end{enumerate}

\item
$\x_n<\y_1-1$ and $\y_n\le \x_1+1$ 

\begin{enumerate}
\item
If $(\x_{n+1},\y_{n+1})=(\x_n-1,\x_n)$ is executed then

\begin{itemize}
\item
$\x_{n+1} = \x_n-1 < \y_1-2<\y_1-1$
\item
$\y_{n+1} = \x_n < \y_1-1<\y_1$
\end{itemize}

Hence (3) holds.

\item
If $(\x_{n+1},\y_{n+1})=(\y_n-2,\x_n+1)$ is executed then 

\begin{itemize}
\item
$\x_{n+1}=\y_n-2 \le \x_1-1< \x_1$
\item
$\y_{n+1} = \x_n < \y_1$
\end{itemize}

Hence (4) holds.
\end{enumerate}

\item
$\x_n<\y_1-1$ and $\y_n < \y_1$ 

\begin{enumerate}
\item
If $(\x_{n+1},\y_{n+1})=(\x_n-1,\x_n)$ is executed then

\begin{itemize}
\item
$\x_{n+1} = \x_n-1 <\y_1-2 <\y_1-1$
\item
$\y_{n+1} = \x_n <\y_1-1<\y_1$.
\end{itemize}

Hence (3) holds.

\item
If $(\x_{n+1},\y_{n+1})=(\y_n-2,\x_n+1)$ is executed then 

\begin{itemize}
\item
$\x_{n+1}=\y_n-2 <\y_1-2 <\y_1-1$
\item
$\y_{n+1} = \x_n <y_1-1 <\y_1$
\end{itemize}

Hence (3) holds.

\end{enumerate}

\item
$\x_n<\x_1$ and $\y_n < \y_1$

\begin{enumerate}

\item
If $(\x_{n+1},\y_{n+1})=(\x_n-1,\x_n)$ is executed then

\begin{itemize}
\item
$\x_{n+1} = \x_n-1 <\x_1 -1<\x_1$
\item
$\y_{n+1} = \x_n <\x_1$
\end{itemize}

Hence (1) holds.

\item
If $(\x_{n+1},\y_{n+1})=(\y_n-2,\x_n+1)$ is executed then 

\begin{itemize}
\item
$\x_{n+1}=\y_n-2 <\y_1-2 < \y_1-1$.
\item
$\y_{n+1} = \x_n <\x_1 <\x_1+1$.
\end{itemize}

Hence (2) holds.
\end{enumerate}

\end{enumerate}

We now have that, for each finite \compseg either
$x,y$, or $x+y$ decreases.

\noindent
{\bf End of Proof of Claim 1}

\bigskip

The following claim is obvious.

\noindent
{\bf Claim 2:}
If any of $x,y$, $x+y$ is 0 then the program terminates.

By  Claims 1 and 2 the premise of Theorem~\ref{th:useramseyf} is satisfied.
Hence Program~\ramsey terminates.
\end{proof}

Consider the following four orderings on $\nat\times\nat$ and the OR of them.

\begin{itemize}
\item
$T_1$ is the ordering
$(\x',\y')<_1 (\x,\y)$ iff 
$\x>0$ and $y>0$ and $\x'<\x$ and $\y' \le \x$.
\item
$T_2$ is the ordering
$(\x',\y')<_2(\x,\y)$ iff 
$\x>0$ and $\y>0$ and $x'<\y-1$ and $\y'\le \x+1$.
\item
$T_3$ is the ordering
$(\x',\y')<_3(\x,\y)$ iff 
$\x>0$ and $\y>0$ and $\x'<\y-1$ and $\y' <\y$.
\item
$T_4$ is the ordering
$(\x',\y')<_4(\x,\y)$ iff 
$\x>0$ and $\y>0$ and $\x'<\x$ and $\y' < \y$.
\item
$T  = T_1 \cup T_2 \cup T_3 \cup T_4.$
We denote this order by $<_T$.
\end{itemize}

Note that (1) each $T_i$ is well-founded, and
(2) for each \compseg 

$$(\x_1,\y_1),(\x_2,\y_2),\ldots,(\x_n,\y_n)$$

\noindent
we have $(\x_1,\y_1)<_T(\x_n,\y_n)$

It is easy to see that these properties of $T$ are all we needed in the proof.
This is Theorem 1 of~\cite{ramseypl} which we state and prove. 

\begin{definition}\label{de:ti}
Let $PROG=(S,I,R)$ be a program. 
\begin{enumerate}
\item
An ordering $T$, which we also denote $<_T$,  on  $S\times S$
is {\it transition invariant} if 
for each \compseg $s_1,\ldots,s_n$ we have
$s_n<_T s_1$.
\item
An ordering $T$ is {\it disjunctive well-founded} if
there exists well-founded orderings $T_1,\ldots,T_k$
such that $T=T_1\cup\cdots\cup T_k$.
Note that the $T_i$ need not be total orderings, they need
only be well-founded. This will come up in the proof of Theorem~\ref{th:prog4}.
\end{enumerate}
\end{definition}

\begin{theorem}\cite{ramseypl}\label{th:usetrans}
Let $PROG=(S,I,R)$ be a program. 
If there exists a disjunctive well-founded transition invariant then
every run of $PROG$ terminates.
\end{theorem}

\begin{proof}
Let $T=T_1\cup\cdots\cup T_k$ be the disjunctive well-founded transition invariant
for $PROG$. Let $<_c$ be the ordering for $T_c$.

Assume, by way of contradiction, that there is an infinite sequence
$s_1,s_2,s_3,\ldots,$ such that each $(s_i,s_{i+1})\in R$.
Define a coloring $COL$ by, for $i<j$, 

\bigskip

\centerline{$COL(i,j) = $ the least $L$ such that $s_j<_L s_i.$}

\bigskip

By Ramsey's Theorem there is an infinite set 

$$i_1 < i_2 < i_3 < \cdots $$

such that

$$COL(i_1,i_2) = COL(i_2,i_3) = \cdots.$$

Let that color be $L$. For notational readability we denote $<_L$ by $<$ and $>_L$ by $>$.
We have

$$s_{i_1} > s_{i_2} > \cdots > $$

This contradicts $<$ being well-founded.
\end{proof}

\begin{note}
It turns out that this theorem is iff. That is, if every computation of $PROG$ is finite
then there is a (perhaps contrived) transition invariant.
\end{note}

Finding an appropriate $T$ is the key to the proofs of termination
for the termination checkers Loopfrog and
Terminator.

The proof of Theorem~\ref{th:usetrans} seems to need the full strength of
Ramsey's Theorem (unlike the proofs of Theorems~\ref{th:useramsey},\ref{th:useramseyf},\ref{th:useramseywf}, 
see the note following its proof). 
In the appendix we give an example, due to Ben-Amram, 
of a program with a disjunctive well-founded transition invariant
where the coloring is not transitive.

If in the premise of Theorem~\ref{th:usetrans} all of the 
$T_i$'s are total (that is, every pair of elements is comparable)
then the transitive Ramsey Theorem suffices for the proof.

\section{Another Proof using Transition Invariants and Ramsey's Theorem}\label{se:sub}

Showing Program~\trramsey terminates
seems easy:
eventually $y$ is negative and after
that point $x$ will steadily decrease until $\x<0$.
But this proof might be hard for a termination checker to find
since $x$ might increases for a very long time.
Instead we need to find the right disjunctive well-founded transition invariant.

\begin{lstlisting}[frame=single,mathescape,title={Program~\trramsey}]
$(x,y) = (\inp(\Z),\inp(\Z))$
While $x>0$
		$(x,y) = (x+y,y-1)$
\end{lstlisting}

\begin{theorem}\label{th:prog4}
Every run of Program~\trramsey terminates.
\end{theorem}

\begin{proof}
We define orderings $T_1$ and $T_2$ which we also denote $<_1$ and $<_2$.
\begin{itemize}
\item
$(\x',\y')<_1(\x,\y)$ iff $0<x'<x$.
\item
$(\x',\y')<_2(\x,\y)$ iff $0\le y'<y$.
\end{itemize}

Let

$$T=T_1\cup T_2.$$

Clearly $T_1$ and $T_2$ are well-founded (though see note after the proof).
Hence $T$ is disjunctive well-founded.
We show that $T$ is a transition invariant.

We want to prove that, for all $n\ge 2$, for each \compsegs of length $n$
$$(\x_1,\y_1), (\x_2,\y_2),\ldots,(\x_n,\y_n)$$
either $(x_n,y_n) <_1 (x_1,y_1)$ or $(x_n,y_n)<_2 (x_1,y_1)$.

We illustrate this with an example. Say $(\x_1,y_1)=(5,4)$. Then the computation will
initially look lie this:

$$(5,4),(9,3),(12,2),(14,1),(15,0)$$

This looks odd since $\x$ is increasing and we want it to be 0. 
but note that

$$(5,4) >_2 (9,3) >_2 (12,2) >_2 (14,1) >_2 (15,0)$$

so the pairs are decreasing in the $<_2$ ordering.

After that the computation looks like this:

$$(15,-1),(14,-2),(12,-3),(9,-4),(5,-5),(0,-6)$$

At this point the computation terminates. We note that

$$(15,-1)>_1(14,-2)>_1(12,-3)>_1(9,-4)>_1(5,-5)>_1(0,-6)$$

Hence in this part of the computation the pairs decrease in the $<_1$ ordering.
Hence the every step of the computation decreases in the $T$ ordering.

By splitting the \compseg

$$(\x_1,\y_1), (\x_2,\y_2),\ldots,(\x_n,\y_n)$$

into two parts depending on if $\y$ is $\ge 0$ or $\y<0$ we can show that
either 
$(\x_1,\y_1) >_1 (\x_n,\y_n)$ or
$(\x_1,\y_1) >_2 (\x_n,\y_n)$, so $(\x_1,\y_1) >_T (\x_n,\y_n)$.
Hence we can apply Theorem~\ref{th:useramseywf} to conclude that the program terminates.
\end{proof}

$T_1$ and $T_2$ are {\it partial orders} not {\it total orders}.
In fact, for both $T_1$ and $T_2$ there are an infinite number
of minimal elements. In particular
\begin{itemize}
\item
the minimal elements for $T_1$ are $\{(\x,\y) \st \x \le 0 \}$, and
\item
the minimal elements for $T_2$ are $\{(\x,\y) \st \y < 0 \}$.
\end{itemize}

\noindent
Recall that the definition of a transition invariant, Definition~\ref{de:ti},
allows partial orders. We see here that this is useful.

\section{Solving Subcases of the Termination Problem}\label{se:dec}

The problem of determining if a program is terminating is unsolvable.
This problem is {\it not} the traditional Halting problem since we allow the
program to have a potentially infinite number of user-supplied inputs.

\begin{definition}~
\begin{enumerate}
\item
Let $M_1^{(\cdots)}, M_2^{(\cdots)},\ldots$ be a standard list of oracle Turing Machines.
These Turing Machines take input in two ways: (1) the standard way, on a tape, and
(2) we interpret the oracle as the user-supplied inputs.
\item
If $A\subseteq \nat$ and $s\in\nat$ then $M_{i,s}^A\cvg$ means that if you run $M_i^A$ (no 
input on the tape) it will halt within $s$ steps.
\item
Let $M_1^{(\cdots)}, M_2^{(\cdots)},\ldots$ be a standard list of oracle Turing Machines.
$$TERM = \{ i \st (\forall A)(\exists s)[M_{i,s}^A\cvg]\}.$$
\end{enumerate}
\end{definition}

\begin{definition}~
\begin{enumerate}
\item
If $A$ and $B$ are subsets of $\nat$ then $A\le_m B$ means that there is a computable function
$f$ such that $x\in A$ iff $f(x)\in B$. (The $m$ is a historical anachronism- it means that $f$
may be many-to-1. There was also a definition $\le_1$ where we insist $f$ be one-to-one. We do not care
anymore, and I personally wonder why anyone ever did.)
\item
$X\in \Pi_1^1$ if there exists an oracle Turing machine $M^{(\cdots)}$ such that
$$X = \{x \st  (\forall A)(\exists x_1)(\forall x_2)\cdots(Q_n x_n)[M^A(x,x_1,\ldots,x_n)=1]\}.$$
($Q_n$ is a quantifier.)
\item
A set $X$ is $\Pi_1^1$-complete if $X\in \Pi_1^1$ and, for all $Y\in \Pi_1^1$, $Y\le_m X$.
\end{enumerate}
\end{definition}

The following were proven by Kleene~\cite{kleene2,kleene} (see also \cite{Rogers}).

\begin{theorem}~
\begin{enumerate}
\item
$X\in \Pi_1^1$ if there exists an oracle Turing machine $M^{(\cdots)}$ such that
$$X = \{x \st  (\forall A)(\exists y)[M^A(x,y)=1]\}.$$
\item
$TERM$ is $\Pi_1^1$-complete.
\item
If $X$ is $\Pi_1^1$-complete then, for all $Y$ in the arithmetic hierarchy,
$Y\le_m X$.
\item
For all $Y$ in the arithmetic hierarchy $Y\le_m TERM$.
This follows from (2) and (3).
(See Definition~\ref{de:ah} for the definition of the Arithmetic Hierarchy.)
\end{enumerate}
\end{theorem}

Hence $TERM$ is much harder than the halting problem.
Therefore it will be very interesting to see if some subcases
of it are decidable.

\begin{definition}\label{de:fun}
Let $n\in \nat$. 
Let $FUN(n)$ be a set of computable functions from $\Z^{n+1}$ to $\Z^n$.
Let $m\in\nat$.
An ($F(n),m)$)-{\it program} is a program of the form of Program~\gen where the functions
$g_i$ used in Program~\gen are all in $FUN(n)$.
\end{definition}

\noindent
{\bf Open Question:}
For which $FUN(n),m$ is the Termination Problem restricted to 
$(FUN(n),m)$-programs decidable?

We list all results we know.
Some are not quite in our framework.
Some of the results use the \While loop condition $Mx \ge b$ where $M$ is a matrix and
$b$ is a vector. Such programs can easily be transformed into
programs of our form.

\begin{enumerate}
\item
Tiwari~\cite{TermLinProgs} has shown that the following problem
is decidable:
Given matrices $A,B$ and vector $c$, all over the rationals,
is Program~\matone in $TERM$. Note that the user is inputting a real.

\begin{lstlisting}[frame=single,mathescape,title={Program~\matone}]
$x  = \inp(\R$)
while ($Bx > b$) 
	$x= Ax+c$
\end{lstlisting}

\item
Braverman~\cite{TermIntLinProg} has shown that the following problem
is decidable:
Given matrices $A,B_1,B_2$ and vectors $b_1,b_2,c$, all over the rationals,
is Program~\mattwo in $TERM$. Note that the user is inputting a real.

\begin{lstlisting}[frame=single,mathescape,title={Program~\mattwo}]
$x = \inp(\R$)
while ($B_1x > b_1$) and ($B_2x \ge b_2$) 
	$x= Ax+c$
\end{lstlisting}

\item
Ben-Amram, Genaim, and Masud~\cite{termintloops} have shown that the following problem is undecidable:
Given matrices $A_0,A_1,B$ and vector $v$ all over the integers,
and $i\in \nat$ does Program~\matthree terminate.
\begin{lstlisting}[frame=single,mathescape,title={Program~\matthree}]
$x =\inp(\Z)$
while ($Bx \ge b$) 
	if $x[i] \ge 0$
		then $x=A_0x$
	else 
		 $x=A_1x$
\end{lstlisting}
\item
Ben-Amram~\cite{BA:delta} has shown a pair of contrasting results:
\begin{itemize}
\item
The termination problem is undecidable for $(FUN(n),m)$-programs
where $m=1$ and 
$FUN(n)$ is the set of all functions of the form

\centerline{$f(x[1],\ldots,x[n])=  \min\{ \x[\ii 1 ]+\cc 1 , \x[\ii 2 ]+\cc 2 , \ldots, \x[\ii k ]+\cc k \}$}

where $1\le \ii 1 < \cdots < \ii k$ and $\cc 1,\ldots,\cc k \in \Z$.
\item
The termination problem is decidable for $(FUN(n),m)$-programs
when $m\ge 1$ and 
$FUN(n)$ is the set of all functions of the form

\centerline{$f(x[1],\ldots,x[n])= x[i]+c$}

where $1\le i  \le n$ and c$\in\Z$.
Note that Program~\ramsey  falls into this category.
\end{itemize}
\item
Joel Ouakine~\cite{integerlinearloops,polyhedral,positive,positive2,orbit} has proven that, for many types of programs that involve
matrices,  it is decidable if the program terminates.
\end{enumerate}

\section{How Much Ramsey Theory Do We Need?}\label{se:need}

Podelski and Rybalchenko~\cite{DBLP:conf/tacas/PodelskiR11}
noted that 
the proofs of Theorems~\ref{th:prog2ramsey},~\ref{th:useramsey}, \ref{th:useramseyf}, and~\ref{th:useramseywf}
do not need the strength of the full Ramsey's Theorem.
In the proofs of these theorems the coloring is transitive.

\begin{definition}
A coloring of the edges of $K_n$ or $K_\nat$ is {\it transitive}
if, for every $i<j<k$, if $COL(i,j)=COL(j,k)$ then both are equal to 
equal $COL(i,k)$.
\end{definition}

\begin{definition}
Let $c,n\ge 1$.
Let $G$ be $K_n$ or $K_\nat$. 
Let $COL$ be a $c$-coloring of the edges of $G$. 
A set of vertices $V$ is a {\it monochromatic increasing  path with respect to $COL$} if
$V=\{v_1<v_2<\cdots \}$ 
and 
$$COL(v_1,v_2)=COL(v_2,v_3)=\cdots.$$
(If $G=K_n$ then the $\cdots$ stop at some $k\le n$.)
We will drop the {\it with respect to $COL$} if the coloring is understood.
We will abbreviate 
{\it monochromatic increasing  path } by \mipit from now on.
\end{definition}

Here is the theorem we really need.
We will refer to it as {\it the Transitive Ramsey's Theorem}.

\begin{theorem}\label{th:infinitees}
Let $c\ge 1$.
For every transitive $c$-coloring of $K_\nat$ 
there exists an infinite \mipns.
\end{theorem}

The Transitive Ramsey Theorem is weaker than Ramsey's Theorem.
We show this in three different ways: (1) Reverse Mathematics,
(2) Computable Mathematics, (3) Finitary Version.

\begin{definition}~
\begin{enumerate}
\item
For all $c\ge 1$ let $RT(c)$ be Ramsey's theorem for $c$ colors.
\item
Let $RT$ be $(\forall c)[RT(c)]$.
\item
For all $c\ge 1$ let $TRT(c)$ be the Transitive Ramsey's theorem for $c$ colors.
\item
Let $TRT$ be $(\forall c)[TRT(c)]$.  (This is the theorem that we really need.)
\end{enumerate}
\end{definition}

\subsection{ Reverse Mathematics}

Reverse Mathematics~\cite{revmath} looks at exactly what strength of
axioms is needed to prove results in mathematics.  A weak axiom system called $RCA_0$ (Recursive Comprehension
Axiom) is at the base. Intuitively a statement proven in $RCA_0$ is proven constructively.

\begin{notation}
\item
Let $A$ and $B$ be statements.
\begin{itemize}
\item
$A\goes B$ means that one  can prove $B$ from $A$ in $RCA_0$.
\item
$A\equiv B$ means that $A\goes B$ and $B\goes A$.
\item
$A\not\goes B$ means that, only using the  axioms in $RCA_0$,  one cannot prove $B$ from $A$.
It may still be the case that $A$ implies $B$ but proving this will require a stronger base axiom system.
\end{itemize}
\end{notation}

The following are known. Items 1 and 2 indicate that the proof-theoretic complexity of $RT$ is 
greater than that of $TRT$.

\begin{enumerate}
\item
$RT\goes TRT$. 
The usual reasoning for this can easily be carried out in $RCA_0$.
\item
Hirschfeldt and Shore~\cite{tramsey} have shown that $TRT \not\goes RT$.
\item
For all $c$, $RT(2) \equiv  RT(c)$. 
The usual reasoning for this can easily be carried out in $RCA_0$.
Note how this contrasts to the next item.
\item
Cholak, Jockusch, and Slaman~\cite{revramsey} showed that $RT(2)\not\goes (\forall c)[RT(c)]$.
\end{enumerate}

The proof of Theorem~\ref{th:prog2ramsey} showed that, over $RCA_0$, 

$$TRT(3) \goes \hbox{Program~\wfNNN terminates}.$$

Does the following hold over $RCA_0$?

$$\hbox{Program~\wfNNN terminates} \goes TRT(3).$$

We do not know. 

In the spirit of the reverse mathematics program we ask the following:
For each $c$ is there a program $P_c$ such that the following holds over $RCA_0$?

$$P \hbox{ terminates } \iff TRT(c).$$

The following is open:
for which $i,j\ge 2$ does $TRT(i)\goes TRT(j)$?

\subsection{Computable Mathematics}\label{se:comp}

Computable Mathematics~\cite{recmath} looks at theorems in mathematics that are proven
non-effectively and questions if there is an effective (that is computable) proof.
The answer is usually no. Then the question arises as to how noneffective the proof is.
Ramsey's Theorem and the Transitive Ramsey's Theorem have been studied and compared in this 
light~\cite{Gasarchcomb,tramsey,Hummel,JockRamsey,SeetSla}.

\begin{definition}\label{de:ah}
Let $M_1^{(\cdots)}, M_2^{(\cdots)},\ldots$ be a standard list of oracle Turing Machines.
\begin{enumerate}
\item
If $A$ is a set then $A'= \{ e \st M_e^A(e)\cvg \}$. This is also called
{\it the Halting problem relative to $A$}. Note that $\es'=HALT$.
\item
A set $A$ is called ${\it low}$ if $A'\le_T HALT$. Note that decidable sets are low.
It is known that there are undecidable sets that are low; however, they have some
of the properties of decidable sets.
\item
We define the levels of the arithmetic hierarchy.
\begin{itemize}
\item
A set is in $\Sigma_0$ and $\Pi_0$ if it is decidable.
\item
Assume $n\ge 1$.
A set $A$ is in $\Sigma_n$ if there exists a set $B\subseteq \nat\times\nat$ that is in $\Pi_{n-1}$
such that 
$$A= \{ x \st (\exists y)[(x,y)\in B]\}.$$
\item
Assume $n\ge 1$. A set $A$ is in $\Pi_n$ if $\overline{A}$ is in $\Sigma_n$.
\item
A set is in the {\it Arithmetic hierarchy} if it is in $\Sigma_n$ or $\Pi_n$ for some $n$.
\end{itemize}
\end{enumerate}
\end{definition}

The following are known. Items 1 and 3 indicate that the Turing degree of the 
infinite homogenous set induced by a coloring is greater than the Turing degree of the
infinite homogenous set induced by a transitive coloring.

\begin{enumerate}
\item
Jockusch~\cite{JockRamsey} has shown that
there exists a computable 2-coloring of the edges of $K_\nat$ such that, for all 
infinite homogeneous sets $H$, $H$ is not computable
in the halting set. 
\item
Jockusch~\cite{JockRamsey} has shown that for every computable  
2-coloring of the edges of $K_\nat$ there exists an infinite homogeneous sets $H\in \Pi_2$.
\item
For all $c$, for every computable transitive $c$-coloring of the edges of $K_\nat$, there exists an infinite \mip $P$ that is computable
in the halting set. This is folklore. 
\item
There exists a computable transitive 2-coloring of the edges of $K_\nat$ with no computable infinite \mip. This is folklore.
\item
Hirschfeldt and Shore~\cite{tramsey} have shown that there exists a computable transitive 2-coloring of
the edges of $K_\nat$ with no infinite low  \mip. 
\end{enumerate}

\subsection{Finitary Version} 

There are finite versions of both Ramsey's Theorem and the Transitive Ramsey's Theorem.
The finitary version of the Transitive Ramsey's Theorem yields better upper bounds.

\begin{notation}
Let $c,k\ge 1$.
\begin{enumerate}
\item
$R(k,c)$ is the least $n$ such that, for any $c$-coloring of the edges of $K_n$,
there exists a homogeneous set of size $k$.
\item
$TRT(k,c)$ is the least $n$ such that, for any transitive 
$c$-coloring of the edges of $K_n$, there exists a \mip of length $k$.
\end{enumerate}
\end{notation}

It is not obvious that $R(k,c)$ and $TRT(k,c)$ exist; however, they do.

The following is well known~\cite{ramseynotes,GRS,RamseyInts} 
and will be prove the $c=2$ case in the appendix.

\begin{theorem}~\label{th:ramseyfinite}
For all $k,c\ge 1$, $c^{k/2}\le R(k,c)\le c^{ck-c+1}$,
\end{theorem}

Improving the upper and lower bounds on the $R(k,c)$ (often called {\it the Ramsey Numbers})
is a long standing open problem. 
The best known asymptotic results for the $c=2$ case are by
Conlon~\cite{ramseyupper}.
For some exact values see Radziszowski's dynamic survey~\cite{ramseysurvey}.

The following theorem is easy to prove; however, neither the statement, nor the proof, seem to
be in the literature. We will prove it in the appendix.

\begin{theorem}~\label{th:tramseyfinite}
For all $k,c\ge 1$ $TRT(k,c)=(k-1)^c+1$.
\end{theorem}

\section{Open Problems}\label{se:open}

\begin{enumerate}
\item
For which ($FUN(n),m$) is the Termination Problem restricted to 
$(FUN(n),m)$-programs decidable?
\item
Find a natural example showing that Theorem~\ref{th:usetrans}
requires the Full Ramsey Theorem.
\item
Prove or disprove that Theorem~\ref{th:usetrans} is equivalent to Ramsey's Theorem.
\item
Classify more types of Termination problems into the classes Decidable and Undecidable.
It would be of interest
to get a more refined classification. Some of the undecidable problems may be equivalent to HALT
while others may be complete in some level of the arithmetic hierarchy or $\Pi_1^1$ complete
\item
Prove or disprove the following conjecture: for every $c$ there is a program $P_c$ such that,
over $RCA_0$, $TRT(c) \iff $ every run of Program $P_c$ terminates.
\end{enumerate}

\section{Summary}

In this survey we discussed various ways to prove that a program always terminates.
The techniques used 
were well-founded orderings, Ramsey Theory, and Matrices.
These techniques work on some programs but not all programs.
We then discussed classes of programs where decidabilty of termination
has been proven. 

The applications of Ramsey Theory only used the transitive Ramsey Theorem.
We discussed the distinction between the two.

Lastly, we listed several open problems.

\section{Acknowledgments}

I would like to thank 
Daniel Apon,
Amir Ben-Amram,
Peter Cholak,
Byron Cook, 
Denis Hirschfeldt,
Jon Katz,
John Ouaknine,
Andreas Podelski, 
Brian Postow,
Andrey Rybalchenko, and
Richard Shore
for helpful discussions.
We would also like to again thank
Amir Ben-Amram for patiently explaining
to me many subtle points that arose
in this paper.
We would also like to thank Daniel Apon for a great
proofreading job.

\appendix

\section{Using Just $C_1$ and $C_2$ to Prove Termination}\label{se:amir}


\begin{definition}
If $\mathcal C$ is a set of square matrices of the same dimension then 
$\clos({\mathcal C})$ is the set of all finite products of elements of $\mathcal C$.
For example, if ${\mathcal C} = \{C_1,C_2\}$ then  $C_1^2C_2C_1^3C_2^{17} \in \clos(C_1,C_2)$.
\end{definition}

This section is due to Ben-Amram and is based on a paper of his~\cite{BA:delta}.
He gives an example of a proof of termination of Program~\ramsey where he uses the matrices $C_1,C_2$
that come out of Program~\ramsey directly (in contrast to our proof in Theorem~\ref{th:prog3matrix} which
used $3\times 3$ matrices by introducing $x+y$).
Of more interest: there {\it is} an element of $\clos(C_1,C_2)$ that has no negative numbers on the diagonal,
namely $C_2$ itself.
Hence we cannot use Theorem~\ref{th:matrix} to prove termination. 

\begin{theorem}\label{th:prog3macho}
Every computation of Program~\ramsey is finite.
\end{theorem}

\begin{proof}

The case control=1  is represented by the matrix
$$ C_1 = \mat{ -1 & 0 \\ \infty & \infty }. $$
The case control=2  is represented by the matrix
$$ C_2 = \mat{ \infty & -2 \\ +1 & \infty }. $$

We find a representation of a {\it superset} of $\clos(C_1,C_2)$.
Let 
$$\mathcal E =  \bigcup_Y {{\mathcal E}_Y} \text{where\ } Y \in \{C_1, C_2, C_1C_2, C_2C_1\}$$
and
$${\mathcal E}_Y =  \{ YZ^a,\ a\ge 1\}. $$
Thus $\mathcal E$ is an infinite set of matrices formed by uniting four classes, each of a simple structure (periodic sets,
in an appropriate sense of the word).
We show that $\clos(C_1,C_2)\subseteq {\mathcal E}$.
We prove this by induction on the number of matrices that are multiplied to form the element of $\clos(C_1,C_2)$.

The base case is trivial since clearly $C_1,C_2 \in {\mathcal E}$.

We show the induction step by 
multiplying each of the four ``patterns'' in $\mathcal E$ \emph{on the left} by each of the matrices $C_1,C_2$.
We use the following identities: $C_1^2 = ZC_1 = C_1Z = C_1C_2$, $C_2^2 = Z$, $ZC_2 = C_2Z$.
\begin{enumerate}
\item $ C_1(C_1Z^a) =  C_1^2 Z^a = C_1Z Z^a = C_1Z^{a+1}$
\item $ C_2(C_1Z^a) =  (C_2C_1)Z^{a}$
\item $ C_1(C_2Z^a) =  (C_1C_2)Z^{a}$
\item $ C_2(C_2Z^a) =  C_2^2Z^{a} = Z Z^a = Z^{a+1}$
\item $ C_1(C_1C_2Z^a) =  C_1^2C_2Z^{a} = C_1 Z C_2 Z^a= C_1 C_2 Z Z^a = (C_1 C_2)Z^{a+1}$
\item $ C_2(C_1C_2Z^a) =  C_2(C_1Z Z^{a}= (C_2C_1)Z^{a+1}$
\item $ C_1(C_2C_1Z^a) =  (C_1 C_2) C_1 Z^{a} = C_1 (Z C_1) Z^a = C_1^2 Z^{a+1} = C_1 Z^{a+2}$
\item $ C_2(C_2C_1Z^a) =  ZC_1Z^{a} = C_1Z^{a+1}$
\end{enumerate}

We have shown that $\clos(C_1,C_2)\subseteq {\mathcal E}$.
Next, we verify that for every class ${\mathcal E}_Y$, either every matrix in ${\mathcal E}_Y$, or every product of
a certain finite number of matrices in ${\mathcal E}_Y$, has a negative integer on the diagonal.  
This suffices for a proof of termination by Theorem~\ref{th:usetrans}, since every class induces a well-founded order
(if an order is not well-founded, every finite power of it is not well-founded either).
The second case occurs here only once (for the second class) and the negative number occurs already for a product
of two such matrices.
\begin{enumerate}
\item $M = C_1Z^a, \text{\ for some $a$\ } \Rightarrow M = \mat{-1-a & -a \\ \infty & \infty} $
\item $M_1 = C_2 Z^{a},\, M_2 = C_2 Z^{b} \Rightarrow M_1M_2 = C_2^2 Z^{a+b} = Z^{a+b+1}$
\item $M = C_1C_2Z^a \Rightarrow M = C_1Z Z^{a} = C_1 Z^{a+1}$
\item $M\in C_2C_1Z^a  \Rightarrow M = \mat{\infty & \infty \\ -3 & -2} Z^a = \mat{\infty & \infty \\ -3 & -2-a} $.
\end{enumerate}

\end{proof}

\section{A Verification that Needs The Full Ramsey Theorey}

The proof of Theorem~\ref{th:usetrans} seems to need the full strength of
Ramsey's Theorem (unlike the proof of Theorem~\ref{th:useramseywf}, see the note
following its proof). 
We give an example, due to Ben-Amram, of a program with a disjunctive well-founded transition invariant
where the coloring is not transitive.
Consider Program not-transitive

\begin{lstlisting}[frame=single,mathescape,title={Program not-transitive}]
$x = \inp(\Z)$
While $x>0$
		$x  = x \div 2$
\end{lstlisting}

It clearly terminates and you can use the transition invariant $\{(\x,\x') \st \x > \x' \}$ to prove it.
This leads to a transitive coloring.
But what if instead your transition-invariant-generator came up with the following rather odd relations instead:
\begin{enumerate}
\item
$T_1 =\{ (\x,\x') \st  \x >  3\x'\}$
\item
$T_2 =\{ (\x,\x') \st  \x >  \x'+1\}$
\end{enumerate}
Note that $T_1 \cup T_2$ is a disjunctive well-founded transition invariant.
We show that the coloring associated to $T_1\cup T_2$ is not transitive.
\begin{itemize}
\item
$COL(4,2)=2$. That is, $(4,2)\in T_2-T_1$.
\item
$COL(2,1)=2$. That is, $(2,1)\in T_2-T_1$.
\item
$COL((4,1)=1$. That is $(4,1)\in T_1$.
\end{itemize}
Hence $COL$ is not a transitive coloring.


\section{Ramsey's Theorem}

Ramsey Theory is a deep branch of combinatorics.
For two books on the sujbect see~\cite{RamseyInts,Ramsey}.

We will present the finite and infinite Ramsey theorem.
and also the finite and infinite transitive Ramsey theorem.
The only theorem used in this paper is the infinite transtive
Ramsey theorem; however, we give you more so you will have
some context. 

It is somewhat remarkable that this branch
of pure math has an application in programming languages.
See \url{www.cs.umd.edu/~gasarch/ramsey/ramsey.html} or
\cite{appramsey} for other applications of Ramsey Theory.
These applications are largely to other theorems in mathematics
or theoretical computer science. Hence one could argue that
the application to proving programs terminate is the first
{\it real} application.

\subsection{If There are Six People at a Party$\ldots$}

The following is well known recreational math problem:

\noindent
{\bf Question:}
Show that if there are six people at a party, either three of them
mutually know each other, or three of them mutually do not know each other.
We call such a set of people {\it homogenous} since they all bear the same relationship
to each other. We will call set of three either homogenous-K (all three pairs know each other) or 
homogenous-DK (none of the pair knows each other).

\medskip

\noindent
{\bf Solution:}
Let the people be $A,B,C,D,E,F$. Look at how $F$ relates to the rest: there must be either
$\ge 3$ that he knows, or $\ge 3$ that he does not know. We will assume that there
are $\ge 2$ that he knows (the other case is similar).

We can assume that $F$ knows $A,B$ and $C$. If any of $A,B,C$ know each other than we
have a homogenous-K set: $F$ and the pair of $A,B,C$ who know each other.
If none of $A,B,C$ know each other than we have a homogenous-DK set: namely $A,B,C$.
(End of Proof)

What if you only had five people at the party? Are you still guanteed a homogenous set?
No: Take $A,B,C,D,E$ where
the following pairs know each other: $(A,B)$, $(B,C)$, $(C,D)$, $(D,E)$, $(E,A)$, and the remaining
pairs do not know each other. We leave it to the reader that in this scenario there is no homogenous set.

What if you want to have a homogenous set of size four? It turns out that if there are 18 people
at a party there must be a homogenous set of size four; however, if there are 17 people at a party
there is a scenario where there is no homogenous set of size four.

What if you want to have a homogenous set of size five? It turns out that if there are 49 people
at a party there must be a homogenous set of size five; however, if there are 43 people at a party
there is a scenario where there is no homogenous set of size five. It is an open problem to determine
the exact number. 
See~\url{http://www.cs.umd.edu/~gasarch/BLOGPAPERS/ramseykings.pdf}
for an interesting take on the problem.

What if you want to have a homogenous set of size $m$? It turns out that if there is a large number $R(m)$
such that if there are $R(m)$ people at a party there must be a homogenous set. We will prove this.

What if you want to have an infinite (countable)  homogenous set? It turns out that there is an infinite 
number of people at a party\footnote{perhaps they all fit because person $i$ is of height $2^{-i}\times6$ feet
and of width $2^{-}i$ feet} then there is an infinite homogenous set. We will prove this.

We will now state this more mathematically and prove the last assertions, though in the reverse order.

\subsection{Notation}

\begin{note}
In the Graph Theory literature there are (at least) two kinds of coloring.
We present them in this note so that if you happen to read the literature
and they are using coloring in a different way then in these notes,
you will not panic.
\begin{itemize}
\item
Vertex Coloring. Usually one says that the vertices of a graph are
$c$-colorable if there is a way to assign each vertex a color,
using no more than $c$ colors, such that no two adjacent vertices
(vertices connected by an edge) are the same color.
Theorems are often of the form `if a graph $G$ has property BLAH BLAH
then $G$ is $c$-colorable' where they mean vertex c-colorable.
We {\bf will not} be considering these kinds of colorings.
\item
Edge Colorings. Usually this is used in the context of Ramsey Theory
and Ramsey-type theorems.
Theorems begin with `for all $c$-coloring of $K_n$ there exists BLAH such that BLAH.
We {\bf will} be considering these kinds of colorings.
\end{itemize}
\end{note}

Lets go back to our party!
We can think of the 6 people as vertices of~$K_6$.
We can color edge $\{i,j\}$ \RED if $i$ and $j$ know
each other, and \BLUE if they do not.

\begin{definition}
Let $n\ge 2$. Then $K_n$ {\it has a homogeenous $K_m$} if
there is a set $V'$ of $m$ vertices (in~$V$) such that
\begin{itemize}
\item 
there is an edge between every pair of vertices in $V'$:\linebreak $\{\{i,j\}\mid \;i,j\in V'\}\subseteq E$
\item 
all the edges between vertices in $V'$ are the same color: there is some $l\in[c]$ such that $COL(\{i,j\})=l$ for all $i,j\in V'$.
\end{itemize}
\end{definition}

\begin{notation}
$\KN$ is the graph ($V,E)$ where
\[
\begin{array}{rl}
V=&\nat\cr
E=&\{ \{x,y\} \mid x,y\in\nat \}\cr
\end{array}
\]
\end{notation}

We now restate our 6-people-at-a-party theorem:

\begin{theorem}\label{th:ramsey2}
Every 2-coloring of the edges of $K_6$ has
a homogenous set of size 3.
\end{theorem}

The {\it finite Ramsey's Theorem}, usually called {\it Ramsey's Theorem}, is as follows:

\begin{theorem}
For all $c$, for all $m$, there exists an $n$ such that
every $c$-coloring of the edges of $K_n$ has
a homogenous set of size $m$.
\end{theorem}

The {\it infinite Ramsey's Theorem} is as follows:

\begin{theorem}
For all $c$,
Every $c$-coloring of the edges of $\KN$ has
an infinite  homogenous set.
\end{theorem}

We need a way to state these theorems more succintcly.
We introduce some notation.

\begin{notation}~
\begin{enumerate}
\item
If $A$ is a set then $\binom{A}{2}$ is the set of all unordered
pairs of distinct elements of $A$.
Note that the phrase {\it for all $c$-colorings of $K_n$} can now be states as
{\it for all $COL:\binom{[n]}{2}\into [c]$}.
\item
$R_c(m)$ is the least $n$ such that for any $c$-coloring of $\binom{[n]}{2}$
there is a homogenous set of size $m$.
$R(m)$ is $R_2(m)$. We have not shown that $R_c(m)$ exists; however,
we will state theorems like {\it $\ldots R_c(m) \le \ldots$} which will mean
that $R_c(m)$ exists and we have a bound for it.
\item
$R_c(\infinity)=\infinity$ means that for any $c$-coloring of $\binom{\nat}{2}$
there is an infinite homogenous set
\end{enumerate}
\end{notation}

In the sections below we state the infinite and finite Ramsey's Theorem 
using this notation.

\subsection{Proof of the Infinite Ramsey Theorem}

We will prove the infinite Ramsey Theorem. We prove this one first for
three reasons

\begin{enumerate}
\item
The infinite one is the only one that we use in this paper.
\item
The infinite one is {\it easier} to prove than the finite one. The combinatorist
Joel Spencer has said {\it infinite combinatorics is easier than finite combinatorics
since all of those messy constants go away.}
\item
We can derive the finite Ramsey Theorem (usually just called {\it Ramsey's Theorem})
from the infinite one. We will present this proof as well two more standard proofs.
\end{enumerate}

\begin{theorem}\label{th:inframsey}
$R(\infinity)=\infinity$.
\end{theorem}

\begin{proof}

Let $COL$ be a 2-coloring of $\KN$.
We define an infinite sequence of vertices, 
$$x_1,x_2,\ldots,$$
 and an infinite sequence of sets of vertices,
$$V_0,V_1,V_2,\ldots,$$ 
that are based on $COL$.

Here is the intuition: Vertex $x_1=1$ has an infinite number of edges coming out of it.
Some are \REDns, and some are \BLUEns.  Hence there are an infinite number of \RED edges coming out of $x_1$, or there are 
an infinite number of \BLUE edges coming out of $x_1$ (or both).  
Let $c_1$ be a color such that $x_1$ has
an infinite number of edges coming out of it that are colored~$c_1$.
Let $V_1$ be the set of vertices $v$ such that $COL(\{v,x_1\})=c_1$.
Then keep iterating this process.

We now describe it formally.

\[
\begin{array}{rl}
V_0  = & \nat\cr
x_1   =& 1 \cr
c_1 =  &
\begin{cases}
\RED & \hbox{ if }|\{ v\in V_0 \mid COL(\{v,x_1\})=\REDns\}| \hbox{ is infinite}\cr
\BLUE & \hbox{ otherwise }\cr
\end{cases} \cr
V_1 = & \{ v \in V_0 \mid COL(\{v,x_1\})=c_1 \} \hbox{ (note that $|V_1|$ is infinite)}\cr
\end{array}
\]

Let $i\ge 2$, and assume that $V_{i-1}$ is defined. We define $x_i$, $c_i$, and $V_i$:

\[
\begin{array}{rl}
x_i =& \hbox{ the least number in $V_{i-1}$} \cr
     & \cr
c_i = &
\begin{cases}
\RED & \hbox{ if } |\{ v\in V_{i-1} \mid COL(\{v,x_i\})=\REDns\}| \hbox{ is infinite}\cr
\BLUE & \hbox{ otherwise}\cr
\end{cases} \cr
V_i = & \{ v \in V_{i-1} \mid COL(\{v,x_i\})=c_i \} \hbox{ (note that $|V_i|$ is infinite)}\cr
\end{array}
\]

How long can this sequence go on for?
Well, $x_i$ can be defined if $V_{i-1}$ is nonempty.
We can show by induction that, for every $i$, $\;V_i$ is infinite.
Hence the sequence
$$x_1,x_2,\ldots,$$
is infinite.

Consider the infinite sequence 
$$c_1, c_2,\ldots$$
Each of the colors in this sequence is either \RED or \BLUEns. Hence there must be 
an infinite sequence $i_1,i_2,\ldots$ such that $i_1<i_2<\cdots$ and
$$c_{i_1} = c_{i_2} = \cdots $$
Denote this color by $c$, and consider the vertices
$$x_{i_1}, x_{i_2},  \cdots $$

It is easy to see they form an infinite homogenous set.

\end{proof}

We leave it as an easy exercise to prove $c$-color case:

\begin{theorem}
$R_c(\infinity)=\infinity$.
\end{theorem}

\subsection{Proof of the Finite Ramsey Theorem from the Infinite Ramsey Theorem}

\begin{theorem}\label{th:finfrinf2}
For every $m\ge 2$, $R(m)$ exists.
\end{theorem}

\begin{proof}
Suppose, by way of contradiction, that there is some $m\ge 2$ such that $R(m)$ does not exist. Then,
for every $n\ge m$, there is some way to color $K_n$ so that there is
no monochromatic $K_m$.
Hence there exist the following:

\begin{enumerate}
\item
$COL_1$, a 2-coloring of $K_m$ that has no monochromatic $K_m$
\item
$COL_2$, a 2-coloring of $K_{m+1}$ that has no monochromatic $K_m$
\item
$COL_3$, a 2-coloring of $K_{m+2}$ that has no monochromatic $K_m$
\item[{}] $\vdots$
\item[$j$.]
$COL_j$, a 2-coloring of $K_{m+j-1}$ that has no monochromatic $K_m$
\item[{}] $\vdots$
\end{enumerate}

We will use these 2-colorings to form a 2-coloring $COL$ of $\KN$
that has no monochromatic $K_m$.

Let $e_1,e_2,e_3,\ldots$ be a list of all unordered pairs
of elements of $\nat$ such that every unordered pair appears exactly once.
We will color $e_1$, then $e_2$, etc.

How should we color $e_1$?
We will color it the way an infinite number of the $COL_i$'s color it.
Call that color $c_1$.
Then how to color $e_2$?
Well, first consider ONLY the colorings that colored $e_1$ with color $c_1$.
Color $e_2$ the way an infinite number of those colorings color it.
And so forth.

We now proceed formally:

\[
\begin{array}{rl}
J_0 = & \nat \cr
      & \cr
COL(e_1) = &
\begin{cases}
\RED & \hbox{ if } |\{ j \in J_0 \mid COL_j(e_1)=\REDns \}| \hbox{ is infinite}\cr
\BLUE & \hbox{ otherwise}\cr
\end{cases} \cr
J_1 = & \{ j\in J_0 \mid COL(e_1)=COL_j(e_1) \}\cr
\end{array}
\]

Let $i\ge 2$, and assume that $e_1,\ldots,e_{i-1}$ have been colored.
Assume, furthermore, that $J_{i-1}$ is infinite and, for every $j\in J_{i-1}$,

\[
\begin{array}{rl}
COL(e_1) =&  COL_j(e_1)\cr
COL(e_2) =&  COL_j(e_2)\cr
\vdots    &  \cr
COL(e_{i-1}) =&COL_j(e_{i-1})\cr
\end{array}
\]

We now color $e_i$:

\[
\begin{array}{rl}
COL(e_i) = &
\begin{cases}
\RED & \hbox{ if } |\{ j \in J_{i-1} \mid COL_j(e_i)=\REDns \}| \hbox{ is infinite}\cr
\BLUE & \hbox{ otherwise}\cr
\end{cases} \cr
J_i = & \{ j\in J_{i-1} \mid COL(e_i)=COL_j(e_i) \}\cr
\end{array}
\]

One can show by induction that, for every $i$, $\;J_i$ is infinite.
Hence this process {\it never} stops. 

\smallskip

\noindent
{\bf Claim:} If $\KN$ is 2-colored with $COL$, then there is no
monochromatic $K_m$.

\smallskip

\noindent
{\bf Proof of Claim:}

Suppose, by way of contradiction, that there is a monochromatic $K_m$.
Let the edges between vertices in that monochromatic $K_m$ be
$$e_{i_1}, \ldots, e_{i_M},$$
where $i_1<i_2< \cdots < i_M$ and $M=\binom{m}{2}$. 
For every $j\in J_{i_M}$, $\;COL_j$ and $COL$ agree on the colors of those edges. Choose $j\in J_{i_M}$ so that all the vertices of the monochromatic $K_m$ are elements of the vertex set of $K_{m+j-1}$. Then $COL_j$ is a 2-coloring of the edges of $K_{m+j-1}$ that has a monochromatic $K_m$, in contradiction to the definition of $COL_j$.

\noindent
{\bf End of Proof of Claim}

Hence we have produced a 2-coloring of $\KN$ that has no monochromatic~$K_m$.
This contradicts Theorem~\ref{th:inframsey}.
Therefore,  our initial supposition---that $R(m)$ does not exist---is false.
\end{proof}

We leave it as an easy exercise to prove $c$-color case:

\begin{theorem}\label{th:finfrominfc}
For all $c$, for all $m$, $R_c(m)$ exists.
\end{theorem}

\subsection{A Direct Proof of the Finite Ramsey's Theorem}

The proof of Ramsey's theorem give for Theorem~\ref{th:finfrinf2}
did not give a bound on $R(m)$. The following proof gives a bound.
It is similar i spirit to the proof of Theorem~\ref{th:inframsey}.

\begin{theorem}\label{th:1finramsey2}
For every $m\ge 2$, $R(m)\le 2^{2m-2}$.
\end{theorem}
\begin{proof}

Let $COL$ be a 2-coloring of $K_{2^{2m-2}}$.
We define a sequence of vertices, 
$$x_1,x_2,\ldots,x_{2m-1},$$ 
and a sequence of sets of vertices,
$$V_0,V_1,V_2,\ldots,V_{2m-1},$$  
\noindent
that are based on $COL$.

Here is the intuition: Vertex $x_1=1$ has $2^{2m-2}-1$ edges coming out of it.
Some are \REDns, and some are \BLUEns.  Hence there are at least $2^{2m-3}$ \RED edges coming out of $x_1$,
or there are at least $2^{2m-3}$ \BLUE edges coming out of $x_1$. 

\medskip
Let $c_1$ be a color such that $x_1$ has
at least $2^{2m-3}$ edges coming out of it that are colored $c_1$.
Let $V_1$ be the set of vertices $v$ such that $COL(\{v,x_1\})=c_1$.
Then keep iterating this process.

We now describe it formally.

\[
\begin{array}{rl}
V_0= & [2^{2m-2}]\cr
x_1 =& 1 \cr
     &   \cr 
c_1 = &
\begin{cases}
\RED & \hbox{ if } |\{ v\in V_0 \mid COL(\{v,x_1\})=\REDns\}|\ge 2^{2m-3}\cr
\BLUE & \hbox{ otherwise}\cr
\end{cases} \cr
V_1 = & \{ v \in V_0 \mid COL(\{v,x_1\})=c_1 \} \hbox{ (note that $|V_1|\ge 2^{2m-3}$)}\cr
\end{array}
\]

Let $i\ge 2$, and assume that $V_{i-1}$ is defined. We define $x_i$, $c_i$, and $V_i$:

\[
\begin{array}{rl}
x_i =& \hbox{ the least number in $V_{i-1}$} \cr
     & \cr
c_i = &
\begin{cases}
\RED & \hbox{ if } |\{ v\in V_{i-1} \mid COL(\{v,x_i\})=\REDns\}|\ge 2^{(2m-2)-i};\cr
\BLUE & \hbox{ otherwise.}\cr
\end{cases} \cr
V_i = & \{ v \in V_{i-1} \mid COL(\{v,x_i\})=c_i \}\hbox{ (note that $|V_i|\ge 2^{(2m-2)-i}$)}\cr
\end{array}
\]

How long can this sequence go on for?
Well, $x_i$ can be defined if $V_{i-1}$ is nonempty.
Note that $$|V_{2m-2}|\ge 2^{(2m-2)-(2m-2)}=2^0=1$$
Thus if $i-1=2m-2$ (equivalently, $i=2m-1$), then $V_{i-1}=V_{2m-2}\ne\emptyset$, but there is no guarantee that $V_i$ $(=V_{2m-1})$ is nonempty.
Hence we can define 
$$x_1,\ldots,x_{2m-1}$$
Consider the colors
$$c_1, c_2,\ldots,c_{2m-2}$$
Each of these is either \RED or \BLUEns. Hence there must be at least $m-1$ of them that are
the same color.  Let $i_1,\ldots,i_{m-1}$ be such that $i_1<\cdots < i_{m-1}$ and
$$c_{i_1} = c_{i_2} = \cdots = c_{i_{m-1}}$$
Denote this color by $c$, and consider the $m$ vertices
$$x_{i_1}, x_{i_2},  \cdots,  x_{i_{m-1}},x_{i_{m-1}+1}$$
To see why we have listed $m$ vertices but only $m-1$ colors, picture the following scenario: You are building a fence row, and you want (say) 7 sections of fence. To do that, you need 8 fence posts to hold it up. Now think of the fence posts as vertices, and the sections of fence as edges between successive vertices, and recall that every edge has a color associated with it.

\medskip
\noindent
{\bf Claim:} The $m$ vertices listed above form a monochromatic $K_m$.

\noindent
{\bf Proof of Claim:}

First, consider vertex $x_{i_1}$. The vertices $$x_{i_2},\ldots, x_{i_{m-1}},x_{i_{m-1}+1}$$ are elements of $V_{i_1}$, hence the edges $$\{x_{i_1},x_{i_2}\},\ldots \{x_{i_1},x_{i_{m-1}}\},\{x_{i_1},x_{i_{m-1}+1}\}$$ are colored with $c_{i_1}$ ($=c$).

Then consider each of the remaining vertices in turn, starting with vertex~$x_{i_2}$. For example, the vertices $$x_{i_3},\ldots, x_{i_{m-1}},x_{i_{m-1}+1}$$ are elements of $V_{i_2}$, hence the edges $$\{x_{i_2},x_{i_3}\},\ldots \{x_{i_2},x_{i_{m-1}}\},\{x_{i_2},x_{i_{m-1}+1}\}$$ are colored with $c_{i_2}$ ($=c$).

\noindent
{\it End of Proof of Claim}
\end{proof}

Note that this is really the same proof as Theorem~\ref{th:inframsey} except that we had to
keep track of the constants. This is an excellent example of Joel Spencer's quote given above.

We leave it as an easy exercise to prove $c$-color case:

\begin{theorem}\label{th:1finramseyc}
For every $c$, $R_c(m)\le c^{cm-c+1}$.
\end{theorem}

\subsection{Another Direct Proof of the Finite Ramsey's Theorem}

We give an alternative proof of the finite Ramsey's theorem that is similar
in spirit to the original 6-people-at-a-party problem and yields slightly better bounds.
slightly better bounds.

Given $m$, we really want $n$ such that every 2-coloring of $K_n$
has a \RED $K_m$ or a  \BLUE $K_m$.  However, it will be useful
to let the parameter for \BLUE differ from the parameter for \REDns.

\begin{notation}
Let $a,b\ge 2$. Let $R(a,b)$ denote the least number, if it exists,
such that every 2-coloring of $K_{R(a,b)}$ has a \RED $K_a$ or
a \BLUE $K_b$.
Note that $R(m)=R(m,m)$.
\end{notation}

We state some easy facts.
\begin{enumerate}
\item 
For all $a,b$, $R(a,b)=R(b,a)$.
\item
For $b\ge 2$, $\;R(2,b)=b$: First, we show that $R(2,b)\le b$. Given any\linebreak 2-coloring of $K_b$, we want a \RED $K_2$ or a \BLUE $K_b$.  Note that a \RED $K_2$
is just a \RED edge. Hence EITHER
there exists one \RED edge (so you get a \RED $K_2$) OR all the edges are \BLUE
(so you get a \BLUE $K_b$). Now we prove that $R(2,b)=b$. If $b=2$, this is obvious. If $\mbox{$b>2$}$, then the all-\BLUE coloring of $K_{b-1}$ has neither a \RED $K_2$ nor a \BLUE $K_b$, hence $R(2,b)\ge b$. Combining the two inequalities ($R(2,b)\le b$ and $R(2,b)\ge b$), we find that $R(2,b)=b$.
\item
$R(3,3)\le 6$. (This is the 6-people-at-a-party theorem.)
\end{enumerate}

We want to show that, for every $n\ge 2$, $\;R(n,n)$ exists.
In this proof, we show something more: that for all $a,b\ge 2$, $\;R(a,b)$ exists.
We do not really care about the case where $a\ne b$, but that case will
help us get our result. This is a situation where proving more than you need
is easier.

\begin{lemma}\label{le:comb}
For all $x,y\ge 1$,
$\binom{x}{y-1} + \binom{x-1}{y-1} = \binom{x}{y}.$
\end{lemma}

\begin{proof}
One could prove this with algebra; however, we will prove it combinatorially.
How many ways are there to choose $y$ people out of $x$?
The answer is of course $\binom{x}{y}$.
We solve it a different way: consider one of the people, named Alice.
If we do not choose Alice then there are $\binom{x}{y-1}$ ways to choose $y$ people.
If we choose Alice then there are $\binom{x-1}{y-1}$ ways to choose $y$ people.
Hence there are 
$\binom{x}{y-1} + \binom{x-1}{y-1}$ was to choose $y$ people.
Hence
$\binom{x}{y-1} + \binom{x-1}{y-1}=\binom{x}{y}$.
\end{proof}

\begin{theorem}\label{th:2finramsey2}~
\begin{enumerate}
\item
For all $a,b\ge 3$: If $R(a-1,b)$ and $R(a,b-1)$ exist, then $R(a,b)$ exists and
$$R(a,b)\le R(a-1,b) + R(a,b-1)$$
\item
For all $a,b\ge 2$, $\;R(a,b)$ exists and $R(a,b)\le \binom{a+b-2}{a-1}$.
\item
For all $m\ge 2$, $R(m)\le \binom{2^{2m}}{\sqrt{m}}$.
\end{enumerate}
\end{theorem}

\begin{proof}

\smallskip

\noindent
{\bf 1:} Assume $R(a-1,b)$ and $R(a,b-1)$ exist.
Let 
$$n=R(a-1,b) + R(a,b-1)$$
Let $COL$ be a 2-coloring of $K_n$, and
let  $x$ be a vertex. Note that there are 
$$R(a-1,b) + R(a,b-1)  - 1$$
edges coming out of $x$ (edges $\{x,y\}$ for vertices $y$).

Let $\NRE$ be the number of red edges coming out of $x$, and let $\NBE$ be the number of blue edges coming out of $x$.
Note that
$$\NRE + \NBE  = R(a-1,b) + R(a,b-1)  - 1$$
Hence either  
$$\NRE \ge R(a-1,b)$$ 
or 
$$\NBE\ge R(a,b-1)$$

\medskip
There are two cases:

\noindent
{\bf Case 1:} $\NRE \ge R(a-1,b)$.  Let 
$$U=\{ y \mid COL(\{x,y\})=\REDns\}$$
$U$ is of size $\NRE \ge R(a-1,b)$.
Consider the restriction of the coloring $COL$ to the edges between vertices in $U$.
Since 
$$|U|\ge R(a-1,b),$$  
this coloring has a \RED $K_{a-1}$ or a
\BLUE $K_b$. 
Within Case 1, there are two cases:
\begin{enumerate}
\item
There is a \RED $K_{a-1}$. Recall that all of the edges in 
$$\{\{x,u\} \mid u\in U\}$$
are \REDns, hence all the edges between elements of the set $U\union \{x\}$ are \REDns, so they form a \RED $K_a$ and WE ARE DONE.
\item
There is a \BLUE  $K_{b}$. 
Then we are DONE.
\end{enumerate}

\noindent
{\bf Case 2:} $\NBE \ge R(a,b-1)$. Similar to Case 1.

\smallskip
\noindent
{\bf 2:} To show that $R(a,b)$ exists and $R(a,b)\le \binom{a+b-2}{a-1}$, we use induction on $n=a+b$. Since $a,b\ge 2$, the smallest value of $a+b$ is 4. Thus $n\ge 4$.

\smallskip

\noindent
{\bf Base Case:} $n=4$. Since $a+b=4$ and $a,b\ge 2$, we must have $a=b=2$.
From part 1, we know that $R(2,2)$ exists and $R(2,2)=2$. Note that 
$$R(2,2)=2 \le  \binom{2+2-2}{2-1}=\binom{2}{1}=2.$$

\smallskip
\noindent
{\bf Induction Hypothesis:} For all $a,b\ge 2$ such that $a+b=n$,
$\;R(a,b)$ exists and $R(a,b)\le \binom{a+b-2}{a-1}$.

\smallskip
\noindent
{\bf Inductive Step:} Let $a,b$ be such that $a,b\ge 2$ and $a+b=n+1$.

By Part 1,  the induction hypothesis, and Lemma~\ref{le:comb} we have

$$R(a,b) \le R(a,b-1)+R(a-1,b) \le \binom{a+b-3}{a-1} + \binom{a+b-3}{a-2}= \binom{a+b-2}{a-1}.$$

\smallskip

\noindent
{\bf 3:} By Part 2 $R(m,m) \le \binom{2m-2}{m-1}$. By Stirling's formula this can be bounded
above by $O(\frac{2^{2m}}{\sqrt m})$.
\end{proof}

We leave it as an easy exercise to prove $c$-color case:

\begin{theorem}\label{th:2finramseyc}
For every $c$, $R_c(a_1,\ldots,a_c)\le \binom{(\sum_{i=1}^c) - c}{a_1! a_2! \cdots a_c!}$.
\end{theorem}

\subsection{Our Last Word on Ramsey Numbers}

The best known asymptotic results for the $c=2$ case are by
Conlon~\cite{ramseyupper} who has shown 
$$R(m) \le \frac{2^{2m}}{m^{c\log s/\log\log s}}.$$
For some exact values of the Ramsey Numbers see Radziszowski's 
dynamic survey~\cite{ramseysurvey}.

What about lower bounds? 
\Erdos found the first nontrivial bound and in the process
invented the probabilitisc method. 

\begin{theorem}
$R(m) \ge \Omega(m2^{m/2})$.
\end{theorem}

\begin{proof}

Let $n=cm2^{m/2}$ where we determine $c$ later.

We need to find a 2-coloring of $\binom{[n]}{2}$
that has no homogenous set of size $n$.
Or do we? We only have to show that such a coloring {\it exists}.

We do the following probabilitistic experiment: for each edge
randomly pick RED or BLUE to color it (the probaility of each is 1/2).
We show that the probability the graph has a homogenous set of
size $m$ is less than one. Hence there exists a coloring with no
homogenous set of size $m$.

The number of colorings is $2^{\binom{n}{2}}$
The number of colorings that have a homogenous set of size $m$ is
bounded above by

$$\binom{n}{m} \times 2 \times 2^{{\binom{n}{2}-\binom{m}{2}}}.$$

Hence the probability that the coloring has a homogenous set of size $m$
is bounded above by

$$
\frac{\binom{n}{m}\times 2\times 2^{{\binom{n}{2}-\binom{m}{2}}}}{2^{\binom{n}{2}}}
=
\frac{\binom{n}{m}\times 2}{2^{-\binom{m}{2}}}
$$

Stirlings formula and algeba show that there is a choice for $m$
where this is less than one. 
\end{proof}

\begin{note}
If the above proof is done 
carefully then $c$ can be taken to be $\frac{1}{e\sqrt{2}}$. 
\end{note}

The probabilistic method is when you show something exists by showing
that the probabiliity that it does not exist is less than one.
It has many applications. See the book by Alon and Spencer~\cite{ASprob}.

\section{The Transitive Ramsey Theorem}

\subsection{A Common Math Competition Problem}

The following problem will likely appear on some math competition in 2014:

\noindent
{\bf Problem:} Find $x$ such that the following hold:
\begin{enumerate}
\item
All sequences of 2014 distinct real numbers has a monotone  subsequence of length $x$.
\item
There exists a sequence of 2014  distince real numbers that has a monotone subsequence of length $x+1$.
\end{enumerate}

\noindent
{\bf Solution:} $x=45$.

\noindent
1) Let $x_1,x_2,\ldots,x_{2014}$ be a sequence of 2014 distinct reals.
Assume, by way of contradiction,  that there is no monotone subsequence of length $45$.

We define a map from $[2014]$  to $[44]\times[44]$ as follows:
Map $x$ to the the ordered pair $(a,b)$ such that
(1) the longest increasing subsequence that ends at $x$ has length $a$.
(2) the longest decreasing subsequence that ends at $x$ has length $h$.

The map is 1-1: Assume, by way of contradiction, that if $i<j$ both map to $(a,b)$.
Assume that $x_i<x_j$ (the case of $x_i>x_j$ is similar).
The longest increasing subsequence that ends at $x_i$ has length $a$.
Since $x_i<x_j$, 
the longest increasing subsequence that ends at $x_j$ has length at least $a+1$.
Hence $j$ does not map to $(a,b)$. Contradiction.
Hence the map is 1-1.

The domain has size $2014$. The range has size $44\times 44=1936$. Hence there is a 1-1
map between a set of size 2014 and a set of size $<2014$, which is a contradiction.

\medskip

\noindent
2) We construct a sequence of length 2025 (longer than we need) that has no monotone
subsequence of length $46$.

Let $y_1 < y_2 < \cdots < y_{45}$ be numbers such that $y_i +46 < y_{i+1}$.

Consider the sequence

$y_1, y_1-1, y_1-2,\ldots, y_1-44,$

$y_2, y_2-1, y_2-2,\ldots, y_2-44,$

$\vdots$

$y_{45}, y_{44}-1, y_{44}-3,\ldots, y_{44}-44.$

This sequence has $45\times 45 = 2025$ elements.
We call each line a block. Within a block the only monotone subsequences are decreasing and are
of length $\le 45$. A monotone subsequence that uses different blocks must use one from each block
and be increasing. Such a sequence must be of length $\le 45$.

This problem and solution are a subcase of a theorem by 
\Erdos and Szekeres~\cite{ErdosSzek}. They showed the following:

\begin{itemize}
\item
For all $k$, for all
sequences of distinct reals of length $(k-1)^2+1$, there is either an increasing monotone subsequence of length $k$
or a decreasing monotone subsequence of length $k$. 
\item
For all $k$, there exists a 
sequences of distinct reals of length $(k-1)^2$ with neither an increasing monotone subsequence of length $k$
or a decreasing monotone subsequence of length $k$. 
\end{itemize}

\subsection{View in terms of Colorings}

Note that we can view a sequence $x_1,\ldots,x_n$ as a 2-coloring of $\binom{[n]}{2}$ via 

\begin{equation}
COL(i<j) =
\begin{cases}
\RED & \hbox{ if $x_i < x_j$ }\\
\BLUE & \hbox{ if $x_i > x_j$ }\\
\end{cases}
\end{equation}

Using Ramsey Theory we would obtain the weak result that there is montone subsequence of length roughly
$\log_2 n$. A modification of the solution above yields a montone subsequence of length roughly $\sqrt n$.
The key is that this is not juste any coloring--- its a transitive coloring. With that in mind we can
generalize the theorem of \Erdos and Szekeres.

\begin{definition}
A {\it transitive $c$-coloring of $\binom{[n]}{2}$ } is a mapping where if $COL(i,j)=COL(j,k)$ then
that color is also $COL(i,k)$.
\end{definition}

\subsection{The Transitive Ramsey Theorem}

\begin{definition}
Let $c\ge 1$ and $n\in\nat \cup \{\nat\}$.
Let $COL$ be a $c$-coloring of $\binom{[n]}{2}$
A set of vertices $V$ is a {\it monochromatic increasing  path with respect to $COL$} if
$V=\{v_1<v_2<\cdots \}$ 
and 
$$COL(v_1,v_2)=COL(v_2,v_3)=\cdots.$$
(If $G=K_n$ then the $\cdots$ stop at some $k\le n$.)
We will drop the {\it with respect to $COL$} if the coloring is understood.
We will abbreviate 
{\it monochromatic increasing  path } by \mipit from now on.
\end{definition}

\begin{definition}
$TRT_c(m)$ is the least $n$ such that any transitive $c$-coloring of $\binom{[n]}{2}$ has
a homogenous set. Note that by Ramsey's theorem (Theorem~\ref{th:1finramseyc}) $TRT_c(m)\le c^{cm-c+1}$.
(Using Theorem~\ref{th:2finramseyc} there is a slightly lower, but still exponential, upper bound.)
We will provide an alternative proof with a much smaller upper bound.
$TRT_c(\infinity)$ can be defined in the obvious say. By Ramsey's Theorem it exists and is $\infinity$.
We will supply an alternative proof that uses less machinery.
\end{definition}

\begin{theorem}\label{th:fintr}
$TRT_c(m)\le (m-1)^c+1$.
\end{theorem}

\begin{proof}

\noindent
1) Let $n=(m-1)^c+1$.
Assume, by way of contradiction, that there is transitive $c$-coloring
of $\binom{[n]}{2}$ 
that has no \mip of length $m$.

We define a map from $\{1,\ldots,n\}$ to $\{1,\ldots,m-1\}^c$ as follows:
Map $x$ to the the vector $(a_1,\ldots,a_c)$ such that
the longest mono path of color $i$ that ends at $x$ has length $a_i$.
Since there are no \mip's of length $m$ the image is a subset of $\{1,\ldots,m-1\}^c$.

It is easy to show that this map is 1-1.
Since $n>(m-1)^c$ this is a contradiction.

\bigskip

\noindent
2) $TRT_c(m) \ge (m-1)^c+1$.

Fix $m\ge 1$. We show by induction on $c$, that, for all $c\ge 1$, there exists a transitive $c$-coloring of 
$\binom{[n]}{2}$ that has no \mip of length $m$. 

\noindent
{\bf Base Case:} $c=1$. We color the edges of $K_{m-1}$ all RED. Clearly there is no \mip of length $m$.

\bigskip

\noindent
{\bf Induction Step:} Assume there is a transitive $(c-1)$-coloring $COL$ of the edges of $K_{(m-1)^{c-1}}$ 
that has no homogeneous set of size $m$. 
Assume that $RED$ is not used. Replace every vertex with a copy of $K_{m-1}$. 
Color edges between vertices in different groups as they were colored by $COL$.
Color edges within a group $RED$. It is easy to see that this produces a transitive $c$-coloring of the edges of
and that there are no \mip of length $m$.
\end{proof}

\begin{theorem}
$TRT_c(\infinity)=\infinity$
\end{theorem}

\begin{proof}
This is similar to the proof of part 1 of Theorem~\ref{th:fintr}.
\end{proof}


\begin{thebibliography}{10}

\bibitem{ASprob}
N.~Alon and J.~Spencer.
\newblock {\em The Probabilistic Method}.
\newblock Wiley, New York, 1992.

\bibitem{BA:delta}
A.~M. Ben-Amram.
\newblock Size-change termination with difference constraints.
\newblock {\em ACM Transactions on Programming Languages and Systems},
  30(3):1--31, 2008.
\newblock \url{http://doi.acm.org/10.1145/1353445.1353450}.

\bibitem{BA:mcs}
A.~M. Ben-Amram.
\newblock Size-change termination, monotonicity constraints and ranking
  functions.
\newblock {\em Logical Methods in Computer Science}, 6(3):1--32, 2010.
\newblock \url{http://www2.mta.ac.il/~amirben/papers.html}.

\bibitem{termintloops}
A.~M. Ben-Amram, S.~Genaim, and A.~N. Masud.
\newblock On the termination of integer loops.
\newblock {\em ACM Transactions on Programming Languages and Systems},
  34(4):1--23, 2012.

\bibitem{TermIntLinProg}
M.~Braverman.
\newblock Termination of integer linear programs.
\newblock In T.~Ball and R.~Jones, editors, {\em Proceedings of the 18th Annual
  International Conference on Computer Aided Verification {\rm Seattle~WA}},
  volume 4144 of {\em Lecture Notes in Computer Science}, pages 372--385, New
  York, 2006. Springer.
\newblock \url{http://www.cs.toronto.edu/~mbraverm/Pub-all.html}.

\bibitem{revramsey}
P.~Cholak, C.~Jockusch, and T.~Slaman.
\newblock On the strength of {R}amsey's {T}heorem for pairs.
\newblock {\em Journal of Symbolic Logic}, 66(1):1--55, 2001.
\newblock \url{http:www.nd.edu/~cholak/papers/}.

\bibitem{orbit}
V.~Chonev, J.~Ouaknine, and J.~Worrell.
\newblock The orbit problem in higher dimensions.
\newblock In {\em STOC '13: Proceedings of the fortyfifth annual ACM symposium
  on Theory of Computing}, pages 80--88, Philadelphia, PA, USA, 2014. Society
  for Industrial and Applied Mathematics.

\bibitem{polyhedral}
V.~Chonev, J.~Ouaknine, and J.~Worrell.
\newblock The polyhedron-hitting problem.
\newblock In {\em SODA '15: Proceedings of the twentysixth annual ACM-SIAM
  symposium on Discrete algorithms}, pages 111--121, Philadelphia, PA, USA,
  2015. Society for Industrial and Applied Mathematics.

\bibitem{ramseyupper}
D.~Conlon.
\newblock A new upper bound for diagonal {R}amsey numbers.
\newblock {\em Annals of Mathematics}, 170(2):941--960, 2009.
\newblock \url{http://www.dpmms.cam.ac.uk/~dc340}.

\bibitem{CPRabs}
B.~Cook, A.~Podelski, and A.~Rybalchenko.
\newblock Abstraction refinement for termination.
\newblock In {\em Static Analysis Symposium (SAS)}, volume 3672 of {\em Lecture
  Notes in Computer Science}, pages 87--101, New York, 2005. Springer.
\newblock \url{http://www7.in.tum.de/~rybal/papers/}.

\bibitem{CPRterm}
B.~Cook, A.~Podelski, and A.~Rybalchenko.
\newblock Termination proofs for systems code.
\newblock In {\em Proceedings of the 2006 ACM SIGPLAN conference on Programming
  language design and implementation}, pages 415--426, New York, 2006. ACM.
\newblock \url{http://www7.in.tum.de/~rybal/papers/}.

\bibitem{proveterm}
B.~Cook, A.~Podelski, and A.~Rybalchenko.
\newblock Proving programs perminate.
\newblock {\em Communications of the ACM}, 54(5):88--97, 2011.
\newblock \url{http://www7.in.tum.de/~rybal/papers/}.

\bibitem{ErdosSzek}
P.~{Erd\H{o}s} and G.~Szekeres.
\newblock A combinatorial problem in geometry.
\newblock {\em Compositio Math}, 2(4):463--470, 1935.
\newblock \url{http://www.renyi.hu/~p\_erodso/1935-01.pdf}.

\bibitem{recmath}
Y.~L. Ershov, S.~S. Goncharov, A.~Nerode, and J.~B. Remmel, editors.
\newblock {\em Handbook of Recursive Mathematics}.
\newblock Elsevier North-Holland, Inc., New York, 1998.

\bibitem{floydpl}
R.~Floyd.
\newblock Assigning meaning to programs.
\newblock In {\em Proceedings of Symposium in Applied Mathematics}, volume~19,
  pages 19--31, Providence, 1967. AMS.
\newblock
  \url{http://www.cs.virginia.edu/~weimer/2007-615/reading/FloydMeaning.pdf }.

\bibitem{Gasarchcomb}
W.~Gasarch.
\newblock A survey of recursive combinatorics.
\newblock In Ershov, Goncharov, Nerode, and Remmel, editors, {\em Handbook of
  Recursive Algebra}, pages 1041--1171. North Holland, 1997.
\newblock \url{http://www.cs.umd.edu/~gasarch/papers/papers.html}.

\bibitem{ramseynotes}
W.~Gasarch.
\newblock Ramsey's theorem on graphs, 2005.
\newblock \url{http://www.cs.umd.edu/~gasarch/mathnotes/ramsey.pdf}.

\bibitem{GRS}
R.~Graham, B.~Rothschild, and J.~Spencer.
\newblock {\em {R}amsey {T}heory}.
\newblock Wiley, New York, 1990.

\bibitem{tramsey}
D.~Hirschfeld and R.~Shore.
\newblock Combinatorial principles weaker than {R}amsey's theorem for pairs.
\newblock {\em Journal of Symbolic Logic}, 72(1):171--206, 2007.
\newblock \url{http://www.math.cornell.edu/~shore/papers.html}.

\bibitem{Hummel}
T.~Hummel.
\newblock Effective versions of {R}amsey's theorem: Avoiding the cone above
  {\bf 0'}.
\newblock {\em Journal of Symbolic Logic}, 59(4):682--687, 1994.
\newblock
  \url{http://www.jstor.org/action/showPublication?journalCode=jsymboliclogic}.

\bibitem{JockRamsey}
C.~Jockusch.
\newblock {R}amsey's theorem and recursion theory.
\newblock {\em Journal of Symbolic Logic}, 37(2):268--280, 1972.
\newblock \url{http://www/jstor.org/pss/2272972}.

\bibitem{kleene}
S.~C. Kleene.
\newblock {\em Introduction to Metamathematics}.
\newblock D. Van Nostrand, Princeton, 1952.

\bibitem{kleene2}
S.~C. Kleene.
\newblock Hierarchies of number theoretic predicates.
\newblock {\em Bulletin of the American Mathematical Society}, 61(3):193--213,
  1955.
\newblock \url{http://www.ams.org/journals/bull/1955-61-03/home.html}.

\bibitem{RamseyInts}
B.~Landman and A.~Robertson.
\newblock {\em {R}amsey Theory on the integers}.
\newblock AMS, Providence, 2004.

\bibitem{Lee:ranking}
C.~S. Lee.
\newblock Ranking functions for size-change termination.
\newblock {\em ACM Transactions on Programming Languages and Systems},
  31(3):81--92, Apr. 2009.
\newblock \url{http://doi.acm.org/10.1145/1498926.1498928}.

\bibitem{LJA}
C.~S. Lee, N.~D. Jones, and A.~M. Ben-Amram.
\newblock The size-change principle for program termination.
\newblock In {\em Proceedings of the 28nd Symposiusm on Principles of
  Programming Languages}, pages 81--92, New York, 2001. ACM.
\newblock \url{http://dl.acm.org/citation.cfm?doid=360204.360210}.

\bibitem{positive}
J.~Ouaknine, J.~Pinto, and J.~Worrell.
\newblock Positivity problems for low-order linear recurrence sequences.
\newblock In {\em SODA '14: Proceedings of the twentyfifth annual ACM-SIAM
  symposium on Discrete algorithms}, pages 90--99, Philadelphia, PA, USA, 2014.
  Society for Industrial and Applied Mathematics.

\bibitem{integerlinearloops}
J.~Ouaknine, J.~S. Pinto, and J.~Worrell.
\newblock On termination of integer linear loops.
\newblock In {\em SODA '15: Proceedings of the twentysixth annual ACM-SIAM
  symposium on Discrete algorithms}, pages 100--110, Philadelphia, PA, USA,
  2015. Society for Industrial and Applied Mathematics.

\bibitem{positive2}
J.~Ouaknine and J.~Worrell.
\newblock On the positivity problem for simple linear recurrence sequences.
\newblock In {\em ICALP '14: Proceedings of the fortyfirst international
  colloquim on automata, languages, and programming}, pages 80--88,
  Philadelphia, PA, USA, 2014. Springer.

\bibitem{PRrank}
A.~Podelski and A.~Rybalchenko.
\newblock A complete method for the synthesis of linear ranking functions.
\newblock In {\em Verification, model checking, and abstract interpretation},
  volume 2937 of {\em Lecture Notes in Computer science}, pages 239--251, New
  York, 2004. Springer.
\newblock \url{http://www7.in.tum.de/~rybal/papers/}.

\bibitem{ramseypl}
A.~Podelski and A.~Rybalchenko.
\newblock Transition invariants.
\newblock In {\em Proceedings of the Nineteenth Annual IEEE Symposium on Logic
  in Computer Science, {\rm Turku, Finland}}, pages 32--41, New York, 2004.
  IEEE.
\newblock \url{http://www7.in.tum.de/~rybal/papers/}.

\bibitem{PRtrans}
A.~Podelski and A.~Rybalchenko.
\newblock Transition predicate abstraction and fair termination.
\newblock In {\em Proceedings of the 32nd Symposiusm on Principles of
  Programming Languages}, pages 132--144, New York, 2005. ACM.
\newblock \url{http://www7.in.tum.de/~rybal/papers/}.

\bibitem{DBLP:conf/tacas/PodelskiR11}
A.~Podelski and A.~Rybalchenko.
\newblock Transition invariants and transition predicate abstraction for
  program termination.
\newblock In P.~A. Abdulla and K.~R.~M. Leino, editors, {\em TACAS}, volume
  6605 of {\em Lecture Notes in Computer Science}, pages 3--10, New York, 2011.
  Springer.
\newblock \url{http://www7.in.tum.de/~rybal/papers/} or
  \url{http://dx.doi.org/10.1007/978-3-642-19835-9_2}.

\bibitem{ramseysurvey}
S.~Radziszowski.
\newblock Small {R}amsey numbers.
\newblock {\em The electronic journal of combinatorics}, 2011.
\newblock \url{www.combinatorics.org}. A dynamic survey so year is last update.

\bibitem{Ramsey}
F.~Ramsey.
\newblock On a problem of formal logic.
\newblock {\em Proceedings of the London Mathematical Society}, 30(1):264--286,
  1930.

\bibitem{Rogers}
H.~{Rogers, Jr.}
\newblock {\em Theory of Recursive Functions and Effective Computability}.
\newblock McGraw Hill, New York, 1967.

\bibitem{appramsey}
V.~Rosta.
\newblock Ramsey theory applications.
\newblock {\em Electronic Journal of Combinatorics}, 13:1--43, 2014.
\newblock This is a dynamic survey.

\bibitem{SeetSla}
D.~Seetapun and T.~A. Slaman.
\newblock On the strength of {R}amsey's {T}heorem.
\newblock {\em Notre Dame Journal of Formal Logic}, 36(4):570--581, 1995.
\newblock
  \url{http://projecteuclid.org/DPubS?service=UI&version=1.0&verb=Display&handle=euclid.ndjfl/1040136917}.

\bibitem{revmath}
S.~G. Simpson.
\newblock {\em Subsystems of Second Order Arithmetic}.
\newblock Springer-Verlag, New York, 2009.
\newblock Perspectives in mathematical logic series.

\bibitem{TermLinProgs}
A.~Tiwari.
\newblock Termination of linear programs.
\newblock In R.~Alur and D.~Peled, editors, {\em Proceedings of the 16th Annual
  International Conference on Computer Aided Verification {\rm Boston~MA},},
  volume 3115 of {\em Lecture Notes in Computer Science}, pages 70--82, New
  York, July 2004. Springer.
\newblock \url{http://www.csl.sri.com/users/tiwari/html/cav04.html}.

\end{thebibliography}

\end{document}